\documentclass[a4paper,11pt]{article}
\usepackage{amsmath,amsfonts,amssymb,amsthm,enumerate}
\usepackage{color}
\def\blue{\textcolor{\blue}}

\renewcommand{\wedge}{\times}

\renewcommand{\leq}{\leqslant}
\renewcommand{\geq}{\geqslant}

\providecommand{\eps}{\varepsilon}

\newcommand{\R}{ { \mathbb{R} } }

\newcommand{\T}{ { \mathbb{T}  } }

\newcommand{\Z}{ { \mathbb{Z} } }

\def\curl{\operatorname{curl}}
\def\Div{\operatorname{div}}

\newcommand{\dt}{ \partial_t }

\newcommand{\tendlorsque}[2]{\mathop{\longrightarrow}\limits_{#1\rightarrow#2}}

\newtheorem{Theorem}{Theorem}
\newtheorem{Definition}[Theorem]{Definition}

\newtheorem{Lemma}[Theorem]{Lemma}
\newtheorem{Remark}{Remark}

\begin{document}

\title{On the weak solutions to the Maxwell-Landau-Lifshitz equations 
 and to the Hall-Magneto-Hydrodynamic equations}

\author{Eric Dumas\footnote{Universit\'e Grenoble 1 - 
Institut Fourier - 100, rue des math\'ematiques - 
BP 74 - 38402 Saint Martin d'H\`eres FRANCE, Support by The Nano-Science Foundation, Grenoble, Project HM-MAG, is acknowledged.}, 
Franck Sueur\footnote{CNRS, UMR 7598, Laboratoire Jacques-Louis Lions, F-75005, Paris, France}
\footnote{UPMC Univ Paris 06, UMR 7598, Laboratoire Jacques-Louis Lions, F-75005, Paris, France}
}
\maketitle

\begin{abstract}
In this paper we deal with weak solutions to the  Maxwell-Landau-Lifshitz equations 
and to the Hall-Magneto-Hydrodynamic equations. 
First we prove that these solutions  satisfy some weak-strong uniqueness property. 
Then we investigate the validity of energy identities.
In particular we give a sufficient condition on the regularity 
of weak solutions to rule out anomalous dissipation. 
In the case of the Hall-Magneto-Hydrodynamic equations we also give 
a sufficient condition to guarantee the  magneto-helicity identity.
Our conditions correspond to the same heuristic scaling as the one introduced by Onsager in hydrodynamic theory.
Finally we examine the sign, locally, of the  anomalous dissipations of weak solutions obtained by some natural approximation processes. 

\end{abstract}



\tableofcontents
\newpage
%%%%%%%%%%%%%%%%%%%%%%%%%%%%%%%%%%%%%%%%%%%%%%%%%%
%%%%%%%%%%%%%%%%%%%%%%%%%%%%%%%%%%%%%%%%%%%%%%%%%%

\section{Introduction}

In this paper we consider two non stationary quasilinear systems of PDEs 
originating from two different physical contexts, for which we develop 
a similar mathematical analysis. These systems are the Maxwell-Landau-Lifshitz  
equations and the  Hall-Magneto-Hydrodynamic equations. 
When studying the associated Cauchy problem, 
weak solutions can be constructed, and they satisfy energy 
inequalities, 
which are equalities if the solutions are smooth.

First we prove that these solutions  satisfy some weak-strong uniqueness property. 
More precisely we state that strong solutions are unique among the class of weak solutions.
A key point here is precisely that the weak solutions considered satisfy some energy 
inequalities.

Then we investigate, according to the regularity of solutions, 
if the energy inequalities are equalities, or if some anomalous dissipation 
shows up during evolution. 
In order to prove that, under some regularity assumption on the solution, 
no dissipation occurs, we essentially analyze the commutation 
between regularization operators and nonlinearities. 
Of course  energy identities rely on the particular structure 
of the nonlinearities of the systems under study, and a regularization 
of the equations may destroy this structure so that some cancellations 
arising in the formal operations leading to the  energy identities are 
not true anymore. 
For both equations only quadratic nonlinearities are involved, when the equations are 
written in conservative form.
In the case of the Hall-Magneto-Hydrodynamic equations we also consider some helicity identities.

Finally we examine, locally, the sign of anomalous dissipations. 
One motivation here is that one expects physical solutions effectively 
dissipate (and do not create) energy, globally as well as locally. 

\subsection{Presentation of the two systems}

We start with a presentation of the two systems. 
In both cases we consider that the underlying physical space is 
the three dimensional euclidian space $\R^3$.
\begin{itemize}
\item 
{\bf The Maxwell-Landau-Lifshitz equations (MLL for short),} 
which describes the coupling between the electromagnetic field 
and a magnetizable medium, see \cite{Brown63} and \cite{LL69} 
for Physics references:
\begin{eqnarray}
\label{LLMGC1}
\dt m  &=&   m \wedge (\Delta m + H) 
- m \wedge \Big( m \wedge  (\Delta m + H)\Big) ,
\\ \label{LLMGC2}
\dt H + \curl  E &=& - \dt m ,
\\ \label{LLMGC3}
\dt E - \curl H &=& 0 , 
\\ \label{LLMGC4}
\Div  E &=& \Div (H+m) = 0 .
\end{eqnarray}
Here $m(t,x)$ stands for the magnetic moment 
and takes  values in the unit sphere $S^2$ of $\R^3$, whereas 
$H(t,x)$ and $E(t,x)$ are respectively the magnetic and electric fields.

  \item
{\bf The Hall-Magneto-Hydrodynamic equations (HMHD for short)} 
from Plasma Physics, see \cite{Light}:
\begin{eqnarray}
\label{0HMHD1}
\dt u +   u \cdot \nabla u + \nabla p &=& (\curl B) \wedge B + \Delta u
\\ \label{0HMHD2}
\Div u &=& 0 ,
\\ \label{0HMHD3}
\dt B - \curl  (u  \wedge  B ) + \curl  \Big(  ( \curl B )  \wedge B \Big) 
&=& \Delta B ,
\\ \label{0HMHD4}
\Div  B &=& 0 ,
\end{eqnarray}
where $u(t,x)$ and $B(t,x)$ are the fluid velocity and magnetic induction. 

The paper \cite{HMHD} provides a derivation of this system from a two-fluid  isothermal Euler-Maxwell system for electrons and ions. 

The system \eqref{0HMHD1}-\eqref{0HMHD4} is a variant of the following 
{\bf Magneto-Hydrodynamic system} with resistance and dissipation ({\bf MHD for short})
\begin{eqnarray}
\label{0MHD1}
\dt u +    u \cdot \nabla u  + \nabla p &=& (\curl B) \wedge B + \Delta u ,
\\ \label{MHD2}
\Div u &=& 0 ,
\\ \label{MHD3}
\dt B - \curl  (u  \wedge  B )  &=&  \Delta B ,
\\ \label{MHD4}
\Div  B &=& 0 .
\end{eqnarray}
Let us stress that the only difference is that the system 
\eqref{0HMHD1}-\eqref{0HMHD4} contains an extra-term in the left-hand-side 
of \eqref{0HMHD3}, $\curl  \big(  ( \curl B )  \wedge B \big) $ 
which takes the Hall effect into account. 
This effect is believed to be the key for understanding the problem of magnetic reconnection which is involved in geomagnetic storms and solar flares, see for instance \cite{HomanGrauer}. The Hall effect has also been studied in connection with the kinematic dynamo problem  \cite{G}.
\end{itemize}

\subsection{A few formal identities}

In both cases existence of global weak solutions is known.
We will recall precisely these results  below, but let us 
emphasize here and now that their proof uses in a crucial way 
some energy bounds. 
Indeed a few formal computations lead to the following energy 
identities.

\paragraph{Energy identity for the MLL equations:}
\begin{equation}
\label{MLLenergy}
\forall T\geq0, \quad \mathcal{E}_{\rm MLL} (T) 
+  \int_{0}^T   \mathcal{D}_{\rm MLL} (t)  \, dt 
= \mathcal{E}_{\rm MLL} (0) ,
\end{equation}
where
\begin{gather*}
\mathcal{E}_{\rm MLL} (t) := 
\int_{\R^3}  \Big(  | E |^{2} (t,x)  +  | H |^{2} (t,x) 
+  | \nabla m |^{2} (t,x) \Big) dx 
 \ \text{ and } \ 
 \mathcal{D}_{\rm MLL} (t) :=   \int_{ \R^3} | \partial_{t} m |^{2} (t,x)  \, dx   .
\end{gather*}

\paragraph{Energy identity for the MHD and HMHD equations:}
\begin{equation}
\label{HMHDenergy}
\forall T\geq0, \quad 
\mathcal{E}_{\rm MHD} (T)  
+  \int_{0}^T \mathcal{D}_{\rm MHD} (t) \,  dt  
  =  \mathcal{E}_{\rm MHD} (0) 
 \ \text{ and } \ 
\mathcal{E}_{\rm HMHD} (T)  
+  \int_{0}^T \mathcal{D}_{\rm HMHD} (t) \,  dt  
  =  \mathcal{E}_{\rm HMHD} (0) ,
\end{equation}
where we denote 
\begin{equation*}
\mathcal{E}_{\rm MHD} (t) = \mathcal{E}_{\rm HMHD} (t)   
:= \frac12 \int_{\R^3}  \Big(   |  u  |^2   +  |  B  |^2  \Big) (t,x)   \, dx  ,
\end{equation*}
\begin{equation*}
\mathcal{D}_{\rm MHD} (t) = \mathcal{D}_{\rm HMHD} (t) 
:= \int_{\R^3} \Big(  | \curl u  |^2    +  | \curl B  |^2   \Big) (t,x)    \, dx  .
\end{equation*}

These two formal identities can be justified when the solutions involved 
are smooth. 
However the weak solutions alluded here are obtained as weak limits 
of smooth functions so that only an inequality can be justified.

\par \ \par 
Another interesting quantity for the MHD and HMHD equations is the magneto-helicity 
\begin{equation*}
\mathcal{H}_m   (t):= \int_{\R^3} \big( B  \cdot A\big)  (t,x) \, dx ,
\end{equation*}
where $A$ is a vector potential of $B$, 
that is a vector field satisfying $\curl A = B$.  
Indeed since $B$ is divergence free the  magneto-helicity  is independent of the choice of the vector potential. In what follows 
we consider  the gauge choice $\Div A = 0$.

\paragraph{Magneto-helicity identity for the MHD and HMHD equations:} 
\begin{equation}
\label{heliFormal}
\forall T\geq0, \quad \mathcal{H}_m (T) 
+  \int_0^{T}  \mathcal{D}_{m} (t)  \,  dt 
= \mathcal{H}_m (0)  ,
\end{equation}
where
\begin{gather*}
 \mathcal{D}_{m} (t) :=  2 \int_{ \R^3}\big( B \cdot (\curl B) \big)  (t,x)  \, dx   .
\end{gather*}

Once again this formal identity can be justified for smooth solutions, 
but not for weak solutions, a priori.

\par \ \par 

One may also consider the fluid helicity
\begin{equation*}
\mathcal{H}_{f}  (t) := \int_{\R^3} \big( u \cdot \omega\big) (t,x) \, dx ,
\end{equation*}
where  $\omega := \curl u $ denotes the vorticity of the fluid. 
One then has formally the following identity.

\paragraph{Fluid helicity  identity  for the MHD and HMHD equations:}
\begin{equation}
\label{fluideHeli}
\forall T\geq0, \quad  \mathcal{H}_{f} (T) 
+   \int_{0}^T \mathcal{D}_{f} (t)    \, dt 
=  \mathcal{H}_{f}  (0) ,
\end{equation}
where
\begin{gather*}
 \mathcal{D}_{f} (t) :=  2 \int_{ \R^3} \big(  \omega  \cdot  \big( (B   \wedge \curl  B  )  +  \curl \omega     \big) (t,x)
  \, dx   .
\end{gather*}

\par \ \par

In the case of the MHD system
 it is interesting to consider the crossed fluid-magneto-helicity:
\begin{equation*}
\mathcal{H}_{fm}  (t) := \int_{\R^3} \big( B \cdot u \big) (t,x) \, dx =  \int_{\R^3}  \big( A \cdot \omega\big) (t,x) \, dx,
\end{equation*}
which satisfies  formally the  following identity. 

\paragraph{Crossed fluid-magneto-helicity identity  for the MHD equations:}
\begin{equation}
\label{fmHeli}
\forall T\geq0, \quad  \mathcal{H}_{fm} (T) 
+   \int_0^T    \mathcal{D}_{fm} (t)    \, dt 
=  \mathcal{H}_{fm}  (0) ,
\end{equation}
where
\begin{gather*}
 \mathcal{D}_{fm} (t) :=  2 \int_{ \R^3} \big(  \omega  \cdot  \big( (B   \wedge \curl  B  )  +  \curl \omega     \big) (t,x)
  \, dx   .
\end{gather*}

\par \ \par

Finally, in the case of the HMHD equations, a rather simple identity is obtained if one considers the total  fluid-magneto-helicity:
\begin{equation*}
\mathcal{H}_{f+m}  (t) :=  \int_{\R^3} \big( u(t,x) + A(t,x)\big) \cdot \big(\omega(t,x)+ B(t,x)\big) \, dx .
\end{equation*}
\paragraph{Total  fluid-magneto-helicity identity  for the HMHD equations:}
\begin{equation}
\label{miHeli}
\forall T\geq0, \quad  
\mathcal{H}_{f+m} (T) 
+   \int_0^T   \mathcal{D}_{f+m} (t)  \, dt 
=  \mathcal{H}_{f+m}  (0) ,
\end{equation}
where
\begin{gather*}
 \mathcal{D}_{f+m} (t) :=
 2  \int_{ \R^3}  \big(  (\omega + B)  \cdot  \curl (\omega + B )\big)(t,x)  \, dx  .
\end{gather*}
%

%%%%%%%%%%%%%%%%%

\subsection{A common structure}

Let us emphasize that the identities 
\eqref{MLLenergy}-\eqref{HMHDenergy}-\eqref{heliFormal}-\eqref{fluideHeli}-\eqref{fmHeli}-\eqref{miHeli}
have the same form: 
\begin{equation}
\label{GenGlobal}
 (\mathcal{E} \text{ or }  \mathcal{H}) (T) 
+   \int_0^T    \mathcal{D} (t)    \, dt 
=  (\mathcal{E} \text{ or }  \mathcal{H})   (0) ,
\end{equation}
where the terms $\mathcal{E}$ or   $\mathcal{H}$ denote respectively various energies 
and  helicities, and the term $ \mathcal{D}$ can be interpreted as some ``dissipation", even if we do not claim anything about its sign in general at this point of the paper.

Actually, these global identities are obtained by space-time integration from local identities of the form:
\begin{equation}
\label{GenLocal}
\partial_t (e  \text{ or } h) + d + \Div f = 0 ,
\end{equation}
where  the terms $e$, $h$ and $d$ denote respectively various energy, helicity and  dissipation densities, and $f$ denotes some flux density. 
The global quantities  $\mathcal{E}$,  $\mathcal{H}$ and $ \mathcal{D}$ are then obtained from  $e$, $h$ and $d$  by integration in space, that is:
\begin{equation}
\label{Int}
\mathcal{E} (t)  :=   \int_{ \R^3}    e   (t,x)    \,  dx      , \quad 
\mathcal{H} (t)   :=    \int_{ \R^3}   h   (t,x)  \,   dx     \,  \text{ and }\, 
\mathcal{D} (t)   :=      \int_{ \R^3}    d (t,x) \,   dx      .
\end{equation}
The appropriate densities will be given explicitly in each case in Section \ref{DEnsi}.

%%%%%%%%%%%%%%%%%

\subsection{Physical motivations}

The investigation  of  the validity of energy or helicity identities for weak solutions to  MLL and MHD equations is quite natural mathematically but is also linked to a few physical motivations that we want to address now. 

\paragraph{The LLM equations.}

Physically, singularities of the magnetization field 
are referred to as Bloch points (see \cite{TGDMS03}). 
The mathematical analysis of these singularities 
has not been performed yet, but they may be 
of the same type as the ones of the heat flow. 
For the heat flow of maps from a manifold to the sphere $S^2$, 
at least when the space variable $x$ belongs to some 
2-dimensional manifold, (space-time) singularities 
of weak solutions correspond to the ``bubbling'' phenomenon, 
\emph{i.e.} the asymptotic convergence of the solution, 
up to renormalization, towards some harmonic map 
(see \cite{HW08} for a review on this topic): 
this ``bubble'' can be interpreted as the precise loss 
of energy of the solution at the singularity. 
The Landau-Lifshitz equation is close to this class of equations, 
in the following sense. 
Considering the simplified case, dropping the gyroscopic term 
$m\wedge\Delta m$ (and with no magnetic field), the equation 
$$
\dt m  =  - m \wedge (m \wedge \Delta m)
$$
may be rewritten, for smooth solutions, 
$$
\dt m  = \Delta m + |\nabla m|^2 m , 
$$
using $|m|^2=1$. This is the heat flow equation for maps with values 
in  the sphere $S^2$, so that the Landau-Lifshitz equation \eqref{LLMGC1} 
(with $H=0$) can be viewed as a perturbation of this heat flow equation 
(on the other hand, the Landau-Lifshitz equation 
$\dt m  =  m \wedge \Delta m$, with no Gilbert dissipation term, is a 
Schr\"odinger map equation, as can be seen thanks to the stereographic 
projection -- see \cite{SSB86}).

\paragraph{The HMHD equations.}

In ideal MHD, instead of considering equation \eqref{0HMHD3} 
or equation \eqref{MHD3}, one considers the equation: 
\begin{equation}
\label{ideal}
\dt B - \curl  (u  \wedge  B ) = 0. 
\end{equation}

If one introduces the flow $\eta$ associated with the divergence free 
fluid velocity field $u$, that is the volume-preserving diffeomorphism $\eta(t)$ obtained by solving the ordinary differential equation:
$\dt \eta = u(t, \eta  ) $ with initial data $ \eta (0,x) = x $, 
equation \eqref{ideal} is (formally) equivalent to 
\begin{gather*}
 B (t,\cdot) = \Big( (D\eta) (t,\cdot) \, B(0,\cdot) \Big) (\eta(t, \cdot  )^{-1}). 
\end{gather*}
%\
This means that any motion of the medium transports the 
magnetic field through a diffeomorphism action preserving 
the relative position of the field lines.
The term  ``frozen-in'' has been coined in this context.
The topological structure of such a  field, including 
its  degree of knottedness,  does not change along time evolution.  
In particular the helicity of a field, which measures the average linking of the field lines, or their
relative winding,  cf. \cite{AK}, is preserved under the
action of a volume-preserving diffeomorphism.
A formal way to visualize this relies on Smirnov' decomposition of divergence free vector fields into elementary solenoids, cf. \cite{Smirnov}.
More precisely if we decompose initially the field $B_0$  into a superposition of  elementary solenoids
\begin{equation*}
\mathcal{S}_0 := 
\int_{\T} (\partial_{s} \tau_0 (s)) \delta(x-\tau_0 (s)) ds ,
\end{equation*}
where $\T$ denotes the one-dimensional torus 
and $\T \ni s \mapsto \tau_0 (s) \in \R^3$ denotes a loop, 
then, when time goes by,  the field $B$ is obtained as a 
superposition of  the elementary solenoids
\begin{equation}
\mathcal{S} := 
\int_{\T} (\partial_{s} \tau(t,s)) \delta(x-\tau(t,s)) ds ,
\end{equation}
where $\T \ni s \mapsto \tau(t,s) \in \R^3$ denotes the loop 
obtained by solving  $\dt  \tau(t,s)= u(t, \tau(t,s)  )$, 
with the initial data  $\tau(0,s)= \tau_0 (s)$. 
Thus one sees that the corresponding loops cannot be unknotted  
without contradicting that the flow is a diffeomorphism.
 
Now if one takes into account the Hall effect, by considering 
the equation 
$$\dt B - \curl  (u  \wedge  B ) + \curl  \Big(  ( \curl B )  \wedge B \Big) = 0,$$
the previous analysis remains true if one substitutes to $\eta$ 
the flow associated with the divergence free vector field 
$u -  \curl B$.

However it appears that for a correct description of 
magnetic reconnection it is necessary to take into account 
the magnetic viscosity, as this is done here by considering  
equation \eqref{0HMHD3}.
This implies that the  field $B$ is not simply transported (as a $2$-form).
In this reconnection process a subtle interplay takes place between  the Hall effect and the magnetic viscosity \cite{HomanGrauer}.

Another motivation for the investigation of the 
topological structure of the magnetic field is that it provides 
obstructions to the full dissipation of the magnetic energy 
in stars or planets. 
In particular it has been shown by Arnold and Khesin \cite{AK} 
that helicity bounds from below  the energy. 
The helicity approach to magnetic energy minoration in terms 
of the topology of magnetic lines has been generalized 
by Freedman and He \cite{FH} by introducing the notion 
of asymptotic crossing number.

\subsection{An analogy with Onsager's conjecture}

The validity of conservation laws for weak solutions  
is a quite general issue in PDEs. 
In particular such a question was raised for incompressible flows 
by Onsager in \cite{lars}.  
The conjecture states that the minimal space regularity needed for a
weak solution to the incompressible  Euler equation to conserve energy 
is $1/3$, that is every weak solution to the Euler equations with 
H\"older continuous velocity of order  $h > 1/3$ does not dissipate energy; 
and conversely, there exists a weak solution to  the incompressible 
Euler equations of smoothness of exactly $1/3$  which does not conserve energy.
\par \ \par 
Concerning the first part of the conjecture, there was a renewal of interest starting with a paper by Eyink \cite{Eyink} who also discussed the connections 
of Onsager's conjecture with phenomenological approaches of fully-developed turbulence.  Soon after this, Constantin, E and Titi gave a simple proof in \cite{CET} that a weak solution $u(t,x)$ of the incompressible Euler equations satisfying the condition
\begin{equation}
\label{cbesov}
\lim_{y \rightarrow 0} \frac{1}{| y | } \int_{0}^{T}  \int_{\R^{3}} | u(t,x) - u(t,x-y) |^{3} \, dx \,  dt = 0 
\end{equation}
verifies the energy equality. 

Let us stress that the condition  \eqref{cbesov} is a Besov type condition rather than a  H\"older one.
Actually the result in  \cite{CET} is stated for a velocity in $L^{3} (0,T ; B^{\alpha}_{3,\infty} (\Omega ))$ with $\alpha > \frac{1}{3}$, but the proof works as well under the slightly weaker condition  \eqref{cbesov}, see also  \cite{CCFS,raoul,romancours}.
\par \ \par 
Regarding the other part of the conjecture, the celebrated works by Scheffer \cite{Scheffer} and Shnirelman \cite{sasha} prove that there
are nontrivial distributional solutions to the Euler equations which are compactly supported
in space and time, and which therefore do not conserve the kinetic energy. 
Recently these results were extended by De Lellis 
 and Szekelyhidi in \cite{lellis0} where they prove that there exist infinitely many compactly supported bounded weak solutions to the incompressible Euler equations.
Consequently  the existence of solutions, better than bounded, but with a regularity slightly weaker than \eqref{cbesov}, which dissipate the kinetic energy,  was  proved in  a series of papers culminating in \cite{Buck}.
 \par \ \par 
On the other hand in \cite{BT}  Bardos and Titi prove that there exist some solutions to the incompressible Euler equations which do not satisfy  \eqref{cbesov}
but  which still preserve the energy. Indeed these solutions are some very explicit shear flows with only $L^2$ regularity. 
 \par \ \par 
The issue of the conservation of helicity for  incompressible flows was 
tackled by  \cite{Chae,CCFS}.
In particular Theorem $4.2$ in \cite{CCFS} proves the validity of helicity conservation for solutions to the incompressible  Euler equation which are in $L^\infty (0,T; H^\frac12 (\R^3 ))$ and satisfy 
\begin{equation}
\label{cbesovheli}
\lim_{y \rightarrow 0} \frac{1}{| y |^2 } 
\int_{0}^{T}  \int_{\R^{3}} | u(t,x) - u(t,x-y) |^{3} \, dx \,  dt = 0 .
\end{equation}
\par \ \par 
We wish to mention that some anisotropic versions of the Onsager conjecture 
were studied by  Caflisch, Klapper and Steele in \cite{CKS} and by Shvydkoy 
in \cite{roman}. 
Furthermore, the papers  \cite{CFS,FT} deal with the first part 
of the Onsager conjecture for the incompressible Navier-Stokes equations 
when the fluid occupies a domain limited by a boundary. 
\par \ \par 
Finally, a related phenomenon is described in the recent papers \cite{DG,DG2} 
by Dascaliuc and  Gruji\'c, who study the energy cascade in the physical space 
for both the Euler and Navier-Stokes equations.

\subsection{Structure of the paper}

In the next section we start with   a reminder 
of the weak theories available for the MLL 
and HMHD equations and we present our results.
Then we state some weak-strong uniqueness results  for these solutions. 
Next we list the local counterparts of the formal identities given in the 
introduction, 
and we establish some local conservation identities 
corresponding to \eqref{MLLenergy}, \eqref{HMHDenergy} 
and \eqref{heliFormal}. These identities include in general, 
i.e. for weak solutions, some anomalous dissipation terms. 
Then we provide a regularity condition, of Besov type, 
which is sufficient for the vanishing of these anomalous 
dissipation terms. 
Finally we investigate the sign of these energy anomalous 
dissipations. 
In Sections \ref{proofWSMLL} and \ref{proofWSHMHD}, we prove 
the weak-strong uniqueness results for the MLL and HMHD equations, 
respectively. 
In Section \ref{Localconservations} we provide the proof 
of the first part of these results (local conservations) for both the MLL 
and HMHD equations. The proof of the other part (vanishing of anomalous dissipations) 
requires a few more technicalities which are given 
in Section \ref{secTech}. Then we provide
in Section \ref{Va1} a regularity condition 
sufficient for the vanishing of anomalous energy 
dissipation for the MLL equation. 
In Section \ref{Va2}, we show the vanishing 
of magneto-helicity and energy anomalous dissipations 
for the HMHD equations under analogous conditions. 
Section \ref{pom} is devoted to the crossed fluid-magneto-helicity  identity for the MHD equations. 
In Section \ref{SectionSuitable} we prove the results stated 
in Section \ref{secresults} about the sign of the energy dissipation 
for weak solutions to  the MLL and HMHD equations 
obtained by standard processes. 
An Appendix is devoted to the proof of a technical Bernstein-type lemma for a time-space Besov space involved in the analysis.

%%%%%%%%%%%%%
%%%%%%%%%%%%%
%%%%%%%%%%%%%
%%%%%%%%%%%%%
%%%%%%%%%%%%%
%%%%%%%%%%%%%
%%%%%%%%%%%%%

\section{Presentation of the results}
\label{secresults}

\subsection{A reminder of the weak theories for the MLL and 
HMHD equations}
\paragraph{Existence of global weak solutions to the MLL equations.} 
Let us first recall that  the  MLL system admits some global weak solutions.
This result relies on the following conservative form of \eqref{LLMGC1}, sometimes referred to as the Gilbert form of the equations:
\begin{equation}
\label{LLGC}
\dt m +  m \wedge \dt m =  
2 \sum_i \partial_i  \Big( m \wedge \partial_i m  \Big) + 2 m \wedge H ,
\end{equation}
where the sum is over $1,2,3$.

For smooth functions, the two equations, \eqref{LLMGC1} and \eqref{LLGC}, are  equivalent, but the  last form is more appropriate  for some $u$ with weak regularity. 
Indeed we have the following result of existence of weak solutions to the MLL equations, see  \cite{Visintin85,CF98,GS} (see also, concerning weak solutions 
for the Landau-Lifshitz equation, the papers \cite{AS,GH}); see also the book 
\cite{GD08}, as well as references therein.

\begin{Theorem}
\label{CF}
Let be given $m_0$ in $L^{\infty} (\R^3 ; \R^3)$ such that 
$|m_0| = 1$ almost everywhere and $\nabla m_0$ is in $L^{2} (\R^3 ; \R^{9})$, 
$E_{0}$ and $H_{0}$ in $L^{2} (\R^3 ; \R^3)$ such that 
$\Div  E_0 = \Div (H_0 + m_{0} )= 0 $.
Then, there exists $(m,E,H) : (0,\infty)\times\R^3 \rightarrow \R^9$ 
such that, for all $T>0$, 
\begin{gather*}
m \in  L^\infty ((0,T) \times \R^3 ;  \R^3) , \quad  |m | = 1 \text{ a.e.}, \quad 
\\   \nabla m \in L^\infty ((0,T); L^{2} (\R^3 ; \R^{9} )) \quad \text{and} \quad  
\dt m \in L^{2} ((0,T) \times \R^3 ;  \R^3)  ,
 \\   (E,H) \in 
L^\infty ((0,T) ; L^{2} (\R^3 ; \R^6))    , 
\end{gather*}
and $(m,E,H)$ is a weak solution to the MLL equations 
\eqref{LLGC}-\eqref{LLMGC2}-\eqref{LLMGC3}-\eqref{LLMGC4} 
on $(0,\infty)\times\R^3$, with initial value $(m_0,E_0,H_0)$.
Moreover, this solution satisfies the following energy inequality, 
\begin{equation}
\label{EIM}
\text{for almost every } T>0, \quad
\mathcal{E}_{\rm MLL} (T) 
+  \int_{0}^T   \mathcal{D}_{\rm MLL} (t)   \, dt 
\leqslant  \mathcal{E}_{\rm MLL} (0) .
\end{equation}
\end{Theorem}

\paragraph{Existence of global weak solutions to the HMHD equations.} 
Let us now tackle the case of the HMHD equations. 
These equations are recast in a conservative form:
\begin{eqnarray}
\label{HMHD1}
\dt u +   \Div (u \otimes u - B\otimes B ) + \nabla p_m &=& \Delta u ,
\\ \label{HMHD2}
\Div u &=& 0 ,
\\ \label{HMHD3}
\dt B - \curl  (u  \wedge  B ) + \curl  \Div (B\otimes B) &=& \Delta B ,
\\ \label{HMHD4}
\Div  B &=& 0 ,
\end{eqnarray}
where $p_m$ denotes the magnetic pressure
\begin{equation*}
p_m := p + \frac{1}{2} | B |^2 .
\end{equation*}
Let
\begin{equation*}
\mathcal{H} := \{ \phi \in L^2 ( \R^3 ; \R^3 ) \mid   \Div \phi = 0   \} 
\quad \text{and} \quad 
\mathcal{V} := \{ \phi \in H^1 ( \R^3;\R^3 ) \mid \Div \phi = 0  \} .
\end{equation*}

In the recent paper \cite{CDL}, Chae, Degond and Liu 
establish the existence of global weak solutions for the
incompressible viscous resistive HMHD model written as follows:

\begin{Theorem}[Chae-Degond-Liu, Acheritogaray-Degond-Frouvelle-Liu] 
\label{ADFL}
Let $u_0$ and $B_0$ be in $\mathcal{H}$. 
Then there exists a global weak solution $(u,B)$ to the HMHD model 
\eqref{HMHD1}-\eqref{HMHD4}, corresponding to these initial data. 
Moreover, for all $T>0$, we have 
\begin{equation}
(u,B)  \in 
\Big( L^\infty (0,T ; \mathcal{H} )\cap L^2 (0,T ; \mathcal{V} ) \Big)^2 ,
\end{equation}
and this solution satisfies the following energy inequality:
\begin{equation}
\label{nrjHALL}
\text{for almost every } T>0, \quad
\mathcal{E}_{\rm HMHD} (T)  
+  \int_0^T  \mathcal{D}_{\rm HMHD} (t)  \,   dt 
\leq  \mathcal{E}_{\rm HMHD} (0) .
\end{equation}
\end{Theorem}
Actually the first result concerning existence of weak solutions 
to the HMHD system is due to 
Acheritogaray, Degond, Frouvelle and Liu \cite{HMHD}, 
who prove Theorem~\ref{ADFL} in a periodic setting.

%%%%%%%%%%%%%
%%%%%%%%%%%%%
%%%%%%%%%%%%%
%%%%%%%%%%%%%
%%%%%%%%%%%%%
%%%%%%%%%%%%%
%%%%%%%%%%%%%

\subsection{Weak-Strong  uniqueness}

One major issue about the weak solutions mentioned above is their uniqueness.
In particular regarding the Landau-Lifshitz equations, non-uniqueness of weak solutions is proved in \cite{AS}.
On the other hand, up to our knowledge, uniqueness of weak solutions to the HMHD equations has not been proved or disproved yet.

Still, one way to get uniqueness results is to consider  stronger solutions. 
Actually, for both the MLL and HMHD equations,  there also exists some results about the local-in-time  existence and uniqueness of strong solutions. 
Let us mention here the papers \cite{CF01a,CF01b}   for the MLL equations and 
\cite{CDL} for the HMHD equations.

Facing these two theories, the weak  one and the strong one, it is natural to wonder  if there is a weak-strong  uniqueness principle. 
Indeed, such a property ensures that the weak theory is an extension of the strong one, rather than a bifurcation. 

The following results provide such properties for both the MLL and HMHD equations. 
In both cases a key point is that weak solutions satisfy an energy inequality. 
This echoes the similar well-known results for the incompressible Navier-Stokes and Euler equations, cf. for example, respectively, \cite{Chemin} and \cite[Proposition 1]{lellis2010}.
Let us also mention here, in this direction, the recent extension to the full Navier-Stokes-Fourier system by \cite{FN}.

Let us warn the reader  that we will not try here to minimize the smoothness of the strong solutions involved in the following statement. 

\paragraph{Weak-Strong  uniqueness for the MLL equations.} Our first result states that a  strong solution is unique among the class of weak solutions, as given by Theorem \ref{CF}.

\begin{Theorem}
\label{WSMLL}
Consider initial data $(m_0,E_0,H_0)$ as in Theorem \ref{CF}, 
and assume moreover that they are smooth.

Finally assume that 
\begin{itemize}
\item $(m_{2},E_{2},H_{2}) : (0,\infty)\times\R^3 \rightarrow \R^9$ is a global weak solution to 
 the MLL equations 
\eqref{LLGC}-\eqref{LLMGC2}-\eqref{LLMGC3}-\eqref{LLMGC4} 
on $(0,T)\times\R^3$, with initial value $(m_0,E_0,H_0)$, as given by 
Theorem \ref{CF}; 
\item  $(m_{1},E_{1},H_{1}) : (0,\infty)\times\R^3 \rightarrow \R^9$ is a smooth 
solution to the MLL equations 
\eqref{LLGC}-\eqref{LLMGC2}-\eqref{LLMGC3}-\eqref{LLMGC4}, up to some time $T>0$, also with the  initial value $(m_0,E_0,H_0)$.
\end{itemize}
Then $(m_{2},E_{2},H_{2}) = (m_{1},E_{1},H_{1})$ on $(0,T)\times\R^3$.
\end{Theorem}
Let us mention that, up
to our knowledge, this result is already new for the Landau-Lifshitz equations:
\begin{equation}
\label{heu}
\dt m +  m \wedge \dt m =  2 m \wedge \Delta m     .
\end{equation}
Actually we will first give a proof of the corresponding result for the global weak solution  to the Landau-Lifshitz equations, as given in  \cite[Theorem 1.4]{AS}, and then we will give the proof of Theorem \ref{WSMLL}.

Let us therefore state here the case of the Landau-Lifshitz equation \eqref{heu}.

\begin{Theorem}
\label{WSLL}
Consider an initial data $m_0$ in $L^{\infty} (\R^3 ; \R^3)$ such that 
$|m_0| = 1$ almost everywhere and  such that $\nabla m_0$ is in $L^{2} (\R^3 ; \R^{9})$.
Assume moreover that  $m_0$ is smooth. 
Finally assume that 
\begin{itemize}
\item $m_{2}$ is a global weak solution  of \eqref{heu} on $(0,\infty)\times\R^3$ satisfying the energy inequality: for almost every $T\geq0$,
\begin{equation}
\label{LLenergy}
J_{\rm LL}  [ m_{2}](T) :=  \big(  \int_{\R^3}   | \nabla m_{2} |^{2}  \, dx \big) (T)  
 +  \int_{0}^T  \int_{\R^3}   |\partial_{t}  m_{2}|^{2}  \, dx \, dt
\leqslant \int_{\R^3}   | \nabla m_{0} |^{2}  \, dx  .
\end{equation}
\item  $ m_{1}$ is a smooth solution to the Landau-Lifshitz equation  \eqref{heu} up to some time $T>0$, with the same initial data $m_{0}$.
\end{itemize}
Then $m_{2} = m_1$ on $(0,T) \times\R^3$.
\end{Theorem}
\paragraph{Weak-Strong  uniqueness for the HMHD equations.} Let us now turn to the case of the HMHD equations.

\begin{Theorem}
\label{WSHMHD}
Consider initial data $u_0$ and $B_0$  in $\mathcal{H}$, 
and assume moreover that they are smooth. 

Finally assume that 
\begin{itemize}
\item  $(u_2 ,B_2)$ is a global weak solution to the HMHD model 
\eqref{HMHD1}-\eqref{HMHD4}, associated with the initial data $(u_0 ,B_0 )$, as in Theorem \ref{WSMLL}; 
\item  $(u_1 ,B_1)$ is a smooth solution the HMHD model 
\eqref{HMHD1}-\eqref{HMHD4} on $(0,T)$, for some $T>0$, also associated with the initial data $(u_0 ,B_0 )$.
\end{itemize}
Then $(u_2 ,B_2) = (u_1 ,B_1)$  on $(0,T)\times\R^3$.
\end{Theorem}

We will prove Theorem  \ref{WSHMHD} in a simplified setting which focuses on the difficulty due to the Hall effect. 
Actually the corresponding statement for the MHD equations is well-known, see for instance \cite{DL}, and 
the extension to the general case by combining the corresponding proof and the proof below for the simplified Hall model is routine.

%%%%%%%%%%%%%
%%%%%%%%%%%%%
%%%%%%%%%%%%%
%%%%%%%%%%%%%
%%%%%%%%%%%%%
%%%%%%%%%%%%%
%%%%%%%%%%%%%

\subsection{Local energy and helicity identities}
\label{DEnsi}

In this section we recall the explicit formulations  of the formal local identities hinted in the introduction. 
We start with the energy identities.

\paragraph{Local energy identity for the MLL equations:} formally, 
\begin{equation}
\label{LLlocalsansa}
\dt e_{\rm MLL} + 
d_{\rm MLL}  + 
\Div f_{\rm MLL} 
= 0  ,
\end{equation}
where
\begin{equation}
\label{edfMLL}
e_{\rm MLL}  :=  |E|^{2} + |H|^{2} + |\nabla m|^{2}  , \quad d_{\rm MLL} :=  |\dt m |^{2} , \quad 
f_{\rm MLL}  := - 2 (   \dt m \cdot \partial_i m  )_{i=1,2,3} + 2  H \wedge E .
 \end{equation}

\paragraph{Local energy identity for the HMHD equations:} formally, 
\begin{equation}
\label{HMHDlocalsansa}
\dt e_{\rm HMHD} + 
d_{\rm HMHD}  + 
\Div f_{\rm HMHD} 
= 0  ,
\end{equation}
where 
\begin{gather}
\label{edfHMHD}
  e_{\rm HMHD} :=  \frac{1}{2} \Big( |u |^{2} +   |B |^{2}   \Big) , \quad 
d_{\rm HMHD} :=  |\curl u  |^2 + |\curl B  |^2 ,
\\ f_{\rm HMHD}  :=  ( \frac12 |  u |^2 + p ) u
 + B \wedge  ( u \wedge B )
 + ( \curl u )  \wedge u
 +  ( \curl B )  \wedge B
 +  \big( ( \curl B )  \wedge B \big)  \wedge B
 .
\end{gather}

Let us mention here that, despite the global energy identities are the same for the HMHD and MHD equations, their local counterparts are different. Indeed, for the MHD equations, one has to drop out the last term in the flux density above. 
\paragraph{Local energy identity for the MHD equations:} formally, 
\begin{equation}
\label{MHDlocalsansa}
\dt e_{\rm MHD} + 
d_{\rm MHD}  + 
\Div f_{\rm MHD} 
= 0  ,
\end{equation}
where 
\begin{gather}
\label{edfMHD1}
  e_{\rm MHD} :=  \frac{1}{2} \Big( |u |^{2} +   |B |^{2}   \Big)  = e_{\rm HMHD}
, \quad 
d_{\rm MHD} :=  |\curl u  |^2 + |\curl B  |^2 = d_{\rm HMHD}  ,
\end{gather}
but
\begin{gather}
\label{edfMHD2}
 f_{\rm MHD}  :=  ( \frac12 |  u |^2 + p ) u
 + B \wedge  ( u \wedge B )
 + ( \curl u )  \wedge u
 +  ( \curl B )  \wedge B
\Big) .
\end{gather}

Let us now tackle the helicity identities.

\paragraph{Local magneto-helicity identity for the HMHD equations:}  formally, 
\begin{equation}
\label{mlocalsansa}
\dt h_{m,{\rm HMHD}} + 
d_{m,{\rm HMHD}}  + 
\Div f_{m,{\rm HMHD}} 
= 0  ,
\end{equation}
where
\begin{gather}
\label{hdfm}
  h_{m,{\rm HMHD}} :=  A \cdot B 
, \,
d_{m,{\rm HMHD}} :=     2 B \cdot  \curl B
  , \,
f_{m,{\rm HMHD}} := ( 2 (u- \curl B ) \wedge B - 2 \curl B - \partial_t A ) \wedge A .
\end{gather}

Once again, despite the global magneto-helicity identities are the same for the HMHD 
and MHD equations, their local counterparts are not the same. Indeed, for the MHD equations, on has to drop out the last term in the flux density above. 

\paragraph{Local magneto-helicity identity for the MHD equations:}  formally, 
\begin{equation}
\label{mlocalsansaMHD}
\dt h_{m,{\rm MHD}} + 
d_{m,{\rm MHD}}  + 
\Div f_{m,{\rm MHD}} 
= 0  ,
\end{equation}
where
\begin{gather}
\label{hdfmMHD1}
  h_{m,{\rm MHD}} :=  A \cdot B = h_{m,{\rm HMHD}}
, \quad 
d_{m,{\rm MHD}} :=     2 B \cdot  \curl B = d_{m,{\rm HMHD}}
  , 
  \end{gather}
but
\begin{gather}
\label{hdfmMHD2}
  f_{m,{\rm MHD}} := ( 2 u \wedge B - 2 \curl B - \partial_t A ) \wedge A .
\end{gather}

\paragraph{Local fluid helicity  identity  for the MHD and HMHD equations:} 
formally, 
\begin{equation}
\label{localfluideHeli}
\dt h_{f}
+ d_{f}  
 + \Div   f_{f} 
  = 0, 
\end{equation}
where
\begin{equation*}
h_{f}   := u \cdot \omega , \,
 d_{f}  := 2 \omega  \cdot (\curl \omega  + B \wedge \curl  B ) , \,
 f_{f}  := (\omega \cdot u ) u + (p + \frac12 u^{2} ) \omega  - u \wedge ( \curl  \omega + B  \wedge \curl  B ) .
\end{equation*}

\paragraph{Local crossed fluid-magneto-helicity identity  for the MHD equations:} formally, 
\begin{gather}
\label{localfm}
\dt h_{fm}
+ d_{fm}  
 + \Div   f_{fm} 
  = 0, 
\end{gather}
where
\begin{gather}
\label{hdffm}
h_{fm} := u \cdot B , \,
 d_{fm}  := 2 \omega \cdot \curl B , \, 
f_{fm} :=   (u\cdot B)u + (p -\frac12 |u |^{2} )B + (\curl u ) \wedge B +  (\curl B ) \wedge u   .
\end{gather}

\paragraph{Local total  fluid-magneto-helicity identity for HMHD equations:} formally, 
\begin{equation}
\label{localmiHeli}
\dt h_{f+m}
+ d_{f+m}  
 + \Div   f_{f+m} 
  = 0, 
\end{equation}
where
\begin{gather*}
h_{f+m} :=  ( u+A) \cdot (\omega+ B) , \, 
d_{f+m} := 2  (\omega+ B) \cdot \curl (\omega+ B) ,\, 
\\  f_{f+m} := \Big( \dt  ( u+A) - 2 u \wedge (\omega+ B)  + 2  \curl (\omega+ B) \Big) \wedge ( u+A)     .
\end{gather*}
%

%%%%%%%%%%%%%
%%%%%%%%%%%%%
%%%%%%%%%%%%%
%%%%%%%%%%%%%
%%%%%%%%%%%%%
%%%%%%%%%%%%%
%%%%%%%%%%%%%

\subsection{Regularization of quadratic terms}
\label{choice}

The above weak solutions are of course solutions on $(0,\infty)\times\R^3$ 
in the distributional sense. In the sequel, for a given such solution, 
we shall compare the difference between the equation, with each term 
regularized, and the same (linear or quadratic) terms obtained from 
the regularization of the solution. We thus introduce some notations. 

Let $\psi \in C^\infty_{\rm c}(\R^3 ; \R)$ be nonnegative, and such that 
$\displaystyle \int_{\R^3} \psi(x) \, dx =1$. For all $\eps\in(0,1)$, 
we define the usual mollifier $\psi^\eps := \eps^{-3} \psi (\cdot/\eps)$. 
Then, for any function $u$ on $\R^3$, we set 
\begin{equation}
\label{notaeps}
u_{\eps}  (x) =   ( \psi^\eps * u ) (x) = 
\int_{\R^{3}} \psi^{\eps} (y)  u(x-y) dy .
\end{equation}

For all $\eps\in(0,1)$ and functions $\phi^1, \phi^2$, we also define 
\begin{eqnarray}
\label{convol0}
\mathcal{A}^{\eps} [\phi^1,\phi^2] := 
( \phi^1 \cdot \phi^2 )_{\eps} - \phi^1_\eps \cdot \phi^2_\eps , \\
\label{convol1}
\mathcal{B}^{\eps} [\phi^1,\phi^2] := 
( \phi^1 \wedge \phi^2 )_{\eps} - \phi^1_\eps \wedge \phi^2_\eps , \\  
\label{convol2}
\mathcal{C}^{\eps} [\phi^1,\phi^2 ] := 
(\phi^1 \otimes \phi^2)_{\eps} - \phi^1_\eps \otimes \phi^2_\eps .
\end{eqnarray}
%

%%%%%%%%%%%%%%%%%%%%%%%
%%%%%%%%%%%%%%%%%%%%%%%
%%%%%%%%%%%%%%%%%%%%%%%
%%%%%%%%%%%%%%%%%%%%%%%
%%%%%%%%%%%%%%%%%%%%%%%
%%%%%%%%%%%%%%%%%%%%%%%

\subsection{Some Besov type conditions}
 
 Our goal is to provide some  sufficient conditions, 
similar to \eqref{cbesov}, which rule out anomalous dissipation 
in the MLL and in the HMHD equations. 

The Fourier transform $\mathcal{F}$ is defined on the space of integrable functions $f \in L^1 (\R^3)$ by 
$(\mathcal{F} f) (\xi) := \int_{\R^3} e^{-2i\pi x\cdot\xi} f(x) dx$, 
and extended to an automorphism of the space $\mathcal{S}'(\R^3)$  
of  tempered distributions, which is the dual of the Schwartz space 
$\mathcal{S}(\R^3)$ of rapidly decreasing functions. 
We consider the following extensions of condition \eqref{cbesov}.
\begin{Definition}[] 
\label{oula}
Let $T>0$, $\alpha \in (0,1)$ and $p,r \in [1,\infty]$. 
We denote by $\mathcal{S}'_h$ the space of tempered 
distributions $u$ on $(0,T) \times \R^3$ such that 
for all $\theta\in C^\infty_{\rm c}(\R^3)$, there holds 
$$
\| \theta(\lambda D)u \|_{L^\infty((0,T)\times\R^3)} 
\tendlorsque{\lambda}{\infty} 0,
$$
where $\theta(D)$ is the Fourier multiplier defined by 
$\theta(D)u=\mathcal{F}^{-1}(\theta\mathcal{F}u)$. 
For every function $u$ on $(0,T) \times \R^3$ 
we define, for $(t,y) \in (0,T) \times (\R^3\setminus\{0\})$, 
\begin{equation*}
f_{\alpha,p} [u] (t,y) := 
\frac{\| u(t,\cdot-y)-u(t,\cdot) \|_{L^{p}(\R^3)}}{| y |^{\alpha} } .
\end{equation*}
We denote 
\begin{itemize}
\item by $\widetilde{L}^r(0,T ; \dot{B}^\alpha_{p,\infty} (\R^3))$, 
the space of functions $u$ on $(0,T) \times \R^3$, 
belonging to $\mathcal{S}'_h$, which satisfy
\begin{equation*}
\sup_{y} \| f_{\alpha,p}[u](\cdot,y) \|_{L^r(0,T)} < \infty  , 
\end{equation*}
equipped with the seminorm 
$$
\| u \|_{\widetilde{L}^r(0,T ; \dot{B}^\alpha_{p,\infty} (\R^3))} :=  
\sup_{y} \| f_{\alpha,p}[u] (\cdot,y) \|_{L^r(0,T)} ; 
$$
\item by $\widetilde{L}^r(0,T ; \dot{B}^{\alpha +1}_{p,\infty}(\R^3))$, 
the subspace of the functions $u$ in   
$\widetilde{L}^r(0,T ; \dot{B}^\alpha_{p,\infty} (\R^3))$ 
which satisfy, for $i=1,2,3$, 
$\partial_{i} u \in \widetilde{L}^r(0,T ; \dot{B}^\alpha_{p,\infty}(\R^3))$, 
equipped with the seminorm 
$$
\| u \|_{\widetilde{L}^r(0,T ; \dot{B}^{\alpha+1}_{p,\infty} (\R^3))} := 
\| u \|_{\widetilde{L}^r(0,T ; \dot{B}^\alpha_{p,\infty} (\R^3))} 
+ \sum_{i=1}^3 
\| \partial_i u \|_{\widetilde{L}^r(0,T ; \dot{B}^\alpha_{p,\infty} (\R^3))}; 
$$
\item by $\widetilde{L}^r(0,T ; \dot{B}^\alpha_{p,c_0} (\R^3))$, 
the subspace of the functions $u$ in 
$\widetilde{L}^r(0,T ; \dot{B}^\alpha_{p,\infty} (\R^3))$  
which satisfy
\begin{equation*}
\| f_{\alpha,p} [u] (\cdot,y) \|_{L^r(0,T)} 
\rightarrow 0 \text{ when } y \rightarrow 0;
\end{equation*}
\item by $\widetilde{L}^r(0,T ; \dot{B}^{\alpha +1}_{p,c_0} (\R^3))$, 
the subspace of the functions $u$ in 
$\widetilde{L}^r(0,T ; \dot{B}^\alpha_{p,c_0} (\R^3))$ 
which satisfy, for $i=1,2,3$, 
$\partial_{i} u \in \widetilde{L}^r(0,T ; \dot{B}^\alpha_{p,c_0}(\R^3))$;
\item by 
$L^r (0,T ; L^p (\R^3))_{\rm loc}$ 
the space of functions $u$ on $(0,T)\times\R^3$ such that 
for all $\chi \in C^\infty_{\rm c}((0,T)\times\R^3)$, $\chi u$ belongs to 
$L^r (0,T ; L^p (\R^3))$. 
\item by 
$\widetilde{L}^r(0,T ; \dot{B}^\alpha_{p,\alpha}(\R^3))_{\rm loc}$, where $\alpha$ holds for $\infty$ or $c_0$, 
the space of functions $u$ on $(0,T)\times\R^3$ such that 
for all $\chi \in C^\infty_{\rm c}((0,T)\times\R^3)$, $\chi u$ belongs to 
$\widetilde{L}^r(0,T ; \dot{B}^\alpha_{p,\alpha}(\R^3))$. 
\end{itemize}
\end{Definition}

In particular, condition \eqref{cbesov} is equivalent to  
$u \in \widetilde{L}^{3} (0,T ; \dot{B}^{1/3}_{3,c_0} (\R^3 ))$.

%Here the equations at stake require a more involved use of the various exponents.

The notation $\widetilde{L}$, rather than $L$, is used to emphasize 
the fact that time integration is performed before taking the supremum in $y$. 
This contrasts with the more classical space $L^r(0,T ; \dot{B}^\alpha_{p,\infty} (\R^3))$, the  space of the functions $u$ on $(0,T) \times \R^3$  which satisfy
 $\| \sup_{y} f_{\alpha,p}[u](\cdot,y) \|_{L^r(0,T)} < \infty  $, 
equipped with the seminorm 
$\| u \|_{L^r(0,T ; \dot{B}^\alpha_{p,\infty} (\R^3))} :=  
 \| \sup_{y} f_{\alpha,p}[u] (\cdot,y) \|_{L^r(0,T)} .$
It is not difficult to see that $ L^r(0,T ; \dot{B}^\alpha_{p,\infty} (\R^3)) \subset \widetilde{L}^r(0,T ; \dot{B}^\alpha_{p,\infty} (\R^3)) $.
 This kind of spaces has been introduced by Chemin and Lerner in \cite{CheminLerner}.

%%%%%%%%%%%%%%%%%%%%%%%
%%%%%%%%%%%%%%%%%%%%%%%
%%%%%%%%%%%%%%%%%%%%%%%
%%%%%%%%%%%%%%%%%%%%%%%
%%%%%%%%%%%%%%%%%%%%%%%
%%%%%%%%%%%%%%%%%%%%%%%

\subsection{Anomalous dissipation for the  MLL equations}
We have the following result.
\begin{Theorem}
\label{LLMonsag}
Let $(m,E,H)$ be a weak solution to the MLL equations 
\eqref{LLGC}-\eqref{LLMGC2}-\eqref{LLMGC3}-\eqref{LLMGC4} 
given by Theorem \ref{CF}. 

 Let $d^\mathfrak{a}_{\rm MLL}$ denotes the local anomalous energy  dissipation for the  MLL equations:
\begin{equation}
\label{LLlocal}
d^\mathfrak{a}_{\rm MLL} :=
\dt e_{\rm MLL} + 
d_{\rm MLL}  + 
\Div f_{\rm MLL}    ,
\end{equation}
where $(e_{\rm MLL} , d_{\rm MLL}  ,  f_{\rm MLL}  )$ is given by \eqref{edfMLL}.

\begin{enumerate} [i)]
\item Then the local anomalous energy  dissipation $d^\mathfrak{a}_{\rm MLL}$ can be obtained as follows. 
Let 
\begin{gather}
\label{eMLLM}
 d^{\mathfrak{a},\eps}_{\rm MLL}  :=
  - \mathcal{B}^{\eps} [m ,  \dt m -2 (H +\Delta m) ] \cdot ( \dt m_\eps  -2 (H_\eps  + \Delta m_\eps)) .
    \end{gather}
Then, 
\begin{equation*}
 d^{\mathfrak{a},\eps}_{\rm MLL}  \rightarrow d^{\mathfrak{a}}_{\rm MLL} \text{ in } \mathcal{D}' \big( (0,\infty)\times\R^3 ; \R \big)   \text{ when } \eps \rightarrow  0 ,
\end{equation*}
and this holds true whatever is the mollifier chosen in Section \ref{choice}.

\item Assume furthermore that $m$ belongs to 
$\widetilde{L}^{3} (0,T ; \dot{B}^\alpha_{p,c_0}(\R^3))_{\rm loc}$
for some $\alpha \in (3/2,2)$ and 
\begin{equation}
\label{ship}
p := \frac{9}{3\alpha - 1} .
\end{equation}
Then the local  anomalous energy dissipation $d_{\rm MLL}^{\mathfrak{a}} $ vanishes.
\end{enumerate}
\end{Theorem}
The first part of the theorem provides a way (actually many, since the choice of the mollifier is arbitrary) to obtain the anomalous dissipation.
The second part provides a sufficient condition on the regularity of the weak solution to guarantee that this anomalous dissipation vanishes.
In this case the local energy identity  \eqref{LLlocalsansa}
     holds true
and the global identity  \eqref{MLLenergy}  as well.

\ \par \

We briefly describe the strategy of the proof of the above theorem. 
Let us consider here, to simplify, the case  of  the Landau-Lifshitz equation \eqref{heu}.
Then the energy identity is formally obtained as follows.  
One  takes the inner product of \eqref{heu} with $\dt m$ and  $\Delta m $ 
to get
\begin{eqnarray}
\label{heuR1}
(\dt m )^{2} &=&  2 (m \wedge \Delta m )  \cdot \dt m ,
 \\ \label{heuR2}
 \dt m \cdot \Delta m  +  (m \wedge \dt m )\cdot \Delta m  &=&   0.
\end{eqnarray}
Observe that the combination $\eqref{heuR1} - 2 \eqref{heuR2}$ yields 
\begin{equation*}
(\dt m )^{2} - 2 \dt m \cdot \Delta m = 0. 
\end{equation*}
Then integrate by parts in space and finally integrate in time 
to obtain the energy identity: for any $T>0$, 
\begin{equation*}
\int_{\R^3}   | \nabla m |^{2} (T,x) dx 
+  \int_{(0,T) \times \R^3} | \partial_{t} m |^{2} \, dx \, dt 
= \int_{\R^3}   | \nabla m |^{2} (0,x) dx  .
\end{equation*}

However some terms involved in this process do not even have a sense 
for weak solutions. In fact, we shall apply a smoothing convolution 
to equation \eqref{heu}, using  the notation $\eps$ in the index 
as in \eqref{notaeps}, and then we shall take the inner product with 
the approximation $\dt m_{\eps}$ and  $\Delta m_{\eps} $.
Still some cancellations, which were trivial in the formal calculations above, 
are not guaranteed anymore, and the point is to be able to get rid of the 
spurious terms. 
For example when we take the inner product of the regularized version of \eqref{heu} with $\Delta m_{\eps} $ we face in particular the expression 
\begin{equation}
\label{termchiant}
\int_{(0,T) \times \R^3} (m \wedge \Delta m )_{\eps}  \cdot \Delta m_{\eps} \, dx \, dt  .
\end{equation}
Actually it appears in the proof below that this term is somehow 
the worse we have to cope with. 
The idea behind Theorem \ref{LLMonsag} is that the regularity assumption 
on $m$ allows to get rid of the term  \eqref{termchiant} when $\eps$ goes 
to $0$. 

A formal argument consists in simply counting that in \eqref{termchiant}, 
there appear four derivatives, in a product of three terms. One can then 
think that the regularity threshold above which integration by parts becomes 
possible is $4/3$. For the Euler equations (where the quantity obtained in 
energy estimates is $((u\cdot\nabla)u)\cdot u$), Onsager's conjecture 
precisely indicated the formal threshold $1/3$, which can be interpreted as the result of one derivative in a product of three terms.
 Here, for the MLL system, 
we are in some sense less able to ``share out'' the derivatives, and 
conclude only for a regularity above $3/2$. 
 
However the couple of exponents $(\alpha,p)$ in Theorem \ref{LLMonsag} 
is critical, in the sense given by Shvydkoy in \cite{romancours}, 
referring to the following dimensional argument. 
Let $M,X,T$ be  respectively some units for magnetic moment, length and time. 
Then the quantity in \eqref{termchiant} has  a dimension equal to 
$X^{-1} \, T M^{3}$.
On the other hand the quantity $\| f_{\alpha,p} [m] (\cdot,y) \|_{L^{r} (0,T)}$ 
from Definition~\ref{oula} has a dimension equal to 
\begin{equation*}
X^{-\alpha}\, \Big( T  (M^{p} X^{3} )^{\frac{r}{p}}  \Big)^{\frac{1}{r}} 
= X^{\frac{3}{p}-\alpha} \,  T^{\frac{1}{r}} \, M .
\end{equation*}
 We would like to control 
the term  \eqref{termchiant} by $ \|  f_{\alpha,p} [m]  (\cdot,y)  \|_{L^{r} (0,T)}^{3}$  
which has a dimension equal to $X^{\frac{9}{p}- 3 \alpha} \, T^{\frac{3}{r}} 
\,  M^{3} $, which provides $r=3$ and  \eqref{ship}.

\par \ \par 
In \cite{BT}  the authors prove that there exist some solutions to the incompressible Euler equations which do not satisfy  \eqref{cbesov}
but  which still preserve the energy. It is quite easy to provide 
a similar result for the MLL equation. 
Indeed, omitting the magnetic field, it is sufficient to consider 
the example which is used in \cite{AS} in order to exhibit a case where 
weak solutions are non unique.

%%%%%%%%%%%%%%%%%%%%%%%
%%%%%%%%%%%%%%%%%%%%%%%
%%%%%%%%%%%%%%%%%%%%%%%
%%%%%%%%%%%%%%%%%%%%%%%
%%%%%%%%%%%%%%%%%%%%%%%
%%%%%%%%%%%%%%%%%%%%%%%

\subsection{Anomalous dissipation for the  HMHD equations}

Let us denote by $K[\cdot] $ the Biot-Savart law in $\R^3$. 
We consider $B$ given by Theorem \ref{ADFL} and $A :=  K[B]$, so that 
$$
A \in  L^\infty (0,T ; \mathcal{V} )\cap L^2 (0,T ; H^2 (\R^3)) , 
\quad \text{and} \quad \curl A = B .
$$
Observe in particular that we consider here the gauge choice $\Div A = 0$.

Our second main result concerns the magneto-helicity conservation 
or dissipation for the  HMHD equations. 
\begin{Theorem}
\label{Hallonsaghel}
Let $(u,B)$ be a solution to the HMHD equations given by Theorem \ref{ADFL}.
Let us denote by $d^\mathfrak{a}_m$ the local magneto-helicity anomalous dissipation:
\begin{gather}
\label{localmagnetohelicity}
d^\mathfrak{a}_m  := \dt  h_m  +  d_m   + \Div f_m      ,
\end{gather}
where 
$(h_m , d_m ,  f_m )$ is given by \eqref{hdfm}.
\begin{enumerate}[i)]
\item  Define 
\begin{equation}
\label{eli}
d_m^{\mathfrak{a},\eps }  := 
2   A_{\eps} \cdot \curl    \mathcal{B}^{\eps} [u,B ]   
-  2   A_{\eps} \cdot  \curl     \Div   \mathcal{C}^{\eps} [B,B  ].  
\end{equation}
Then
\begin{equation*}
 d_m^{\mathfrak{a},\eps}  \rightarrow d_m^{\mathfrak{a}} \text{ in } \mathcal{D}' \big( (0,\infty)\times\R^3 ; \R \big)   \text{ when } \eps \rightarrow  0 ,
\end{equation*}
and this holds true whatever is the mollifier chosen in Section \ref{choice}.
\item Assume furthermore that  $B \in 
\widetilde{L}^{3} (0,T ; \dot{B}^{\frac{1}{3}}_{3,c_0} (\R^3))_{\rm loc}$.
Then the anomalous magneto-helicity dissipation $d^\mathfrak{a}_m $ vanishes.
\end{enumerate}
\end{Theorem}
\begin{Remark}
The proof of part ii) below will show that the first term in the definition \eqref{eli} of $d_m^{\mathfrak{a},\eps}$ converges to $0$ when $\eps \rightarrow 0^+$ in the sense of distributions under the sole assumption that $(u,B)$ is given by Theorem \ref{ADFL}. 
\end{Remark}
Investigation of  the validity of the fluid helicity identity 
\eqref{fluideHeli} and of the total  fluid-magneto-helicity 
\eqref{miHeli} is very similar and is left aside in this paper. 
Regarding energy, we have the following result.
\begin{Theorem}
\label{Hallonsagen}
Let $(u,B)$ be a solution to the HMHD equations given by Theorem \ref{ADFL}, 
and assume that $B \in L^{4} ((0,T) \times \R^{3} )_{\rm loc}$. 
Let us denote by $d^\mathfrak{a}_{\rm HMHD} $ the local energy anomalous dissipation:
\begin{equation}
\label{HMHDlocalEps}
d^\mathfrak{a}_{\rm HMHD} :=
\dt e_{\rm HMHD} + 
d_{\rm HMHD}  + 
\Div f_{\rm HMHD}    ,
\end{equation}
where $(e_{\rm HMHD} , d_{\rm HMHD}  ,  f_{\rm HMHD}  )$ is given by \eqref{edfHMHD}.

\begin{enumerate}[i)]
\item Let
\begin{eqnarray*}
 d^{\mathfrak{a},\eps}_{\rm HMHD}  &:=& 
- u_\eps \cdot   \Div  \Big( \mathcal{C}^{\eps} [u,u ] 
- \mathcal{C}^{\eps} [B,B  ]     \Big)  
-   \frac{1}{2} u_\eps \cdot  \nabla \mathcal{A}^{\eps} [B, B ]
\\ &&- B_{\eps} \cdot \Big(  \curl   \mathcal{B}^{\eps} [u,B ] \Big)
+ B_{\eps} \cdot  \Big(  \curl   \Div \mathcal{C}^{\eps} [B,B  ]   \Big)  .
\end{eqnarray*}
Then
\begin{equation*}
  d^{\mathfrak{a},\eps}_{\rm HMHD}    \rightarrow d^\mathfrak{a}_{\rm HMHD}   \text{ in } \mathcal{D}' \big( (0,\infty)\times\R^3 ; \R \big)   \text{ when } \eps \rightarrow  0 ,
\end{equation*}
and this holds true whatever is the mollifier chosen in Section \ref{choice}.
\item Assume furthermore that  
$$u \in  
\widetilde{L}^{3} (0,T ; \dot{B}^{\frac{1}{3}}_{3,c_0} (\R^3))_{\rm loc}
\text{ and }
B \in 
\widetilde{L}^{3} (0,T ; \dot{B}^{\frac{2}{3}}_{3,c_0} (\R^3))_{\rm loc}.
$$
Then the anomalous energy dissipation $d^\mathfrak{a}_{\rm HMHD} $ vanishes.
\end{enumerate}
\end{Theorem}
%

%%%%%%%%%%%%%%%%%%%%%%%
%%%%%%%%%%%%%%%%%%%%%%%
%%%%%%%%%%%%%%%%%%%%%%%
%%%%%%%%%%%%%%%%%%%%%%%
%%%%%%%%%%%%%%%%%%%%%%%
%%%%%%%%%%%%%%%%%%%%%%%

\subsection{A comparison with the MHD equations}
It is interesting to compare the results 
for the  HMHD equations with the case of the 
Magneto-Hydrodynamic system  without Hall effect.
Recasting \eqref{0MHD1} under the conservative form 
\begin{equation}
\label{MHD1}
\dt u +   \Div (u \otimes u - B\otimes B ) + \nabla p_m = \Delta u ,
\end{equation}
we can define the notion of weak solution. 
Actually this was observed a long time ago by Duvaut and Lions who proved 
the following theorem (see \cite{DL}).
\begin{Theorem}[Duvaut-Lions] 
\label{DL}
Let $u_0$ and $B_0$ be in $\mathcal{H}$ and $T>0$. 
Then there exists a weak solution 
\begin{equation*}
(u,B)  \in \Big( L^\infty (0,T ; \mathcal{H} )\cap L^2 (0,T ; \mathcal{V} ) \Big)^2 ,
\end{equation*}
for the MHD model \eqref{MHD1}-\eqref{MHD2}-\eqref{MHD3}-\eqref{MHD4} corresponding to these initial data.
Moreover, this solution satisfies \eqref{nrjHALL}.
\end{Theorem}
Now following the proofs of Theorem \ref{Hallonsaghel} 
and Theorem \ref{Hallonsagen} we also have the following result 
about the conservation of energy,  of magneto-helicity and of crossed fluid-magneto-helicity.
\begin{Theorem}
\label{MHDonsag}
Let $(u,B)$ be a solution to \eqref {MHD1} given by Theorem \ref{DL}.
\begin{enumerate}[i)]
\item  Then  the  local  magneto-helicity identity \eqref{mlocalsansaMHD} is valid.
\item  Assume furthermore that 
\begin{itemize}
\item $u$ is in $\widetilde{L}^{3} (0,T ; \dot{B}^{\alpha}_{3,\infty} (\R^3 ))_{\rm loc}$ with  $\alpha \in (0,1)$ or $u$ is in $L^{3} ((0,T) \times \R^3 )_{\rm loc}$, and then we set $\alpha =0$,
\item $B$ is in $\widetilde{L}^{3} (0,T ; \dot{B}^{\beta}_{3,\infty} (\R^3 ))_{\rm loc}$ with  $\beta \in (0,1)$,
\end{itemize}
and at least one the three following properties holds true:
\begin{itemize}
\item $\alpha=0$, 
\item $\alpha\in(0,1)$ and 
$u$ is in $\widetilde{L}^{3} (0,T ; \dot{B}^{\alpha}_{3,c_{0}} (\R^3 ))_{\rm loc}$, 
\item $B$ is in $\widetilde{L}^{3}   (0,T ; \dot{B}^{\beta}_{3,c_{0}} (\R^3 ))_{\rm loc}$.
\end{itemize}
Finally assume that  $\alpha + 2 \beta \geq 1$.
Then  the  local energy estimate    \eqref{MHDlocalsansa} holds true.
\item Let us consider again $(u,B)$  a solution to \eqref {MHD1} given by Theorem \ref{DL}.
Assume furthermore that 
\begin{itemize}
\item $u$ is in $\widetilde{L}^{3} (0,T ; \dot{B}^{\alpha}_{3,\infty} (\R^3 ))_{\rm loc}$ with  $\alpha \in (0,1)$,
\item $B$ is in $\widetilde{L}^{3} (0,T ; \dot{B}^{\beta}_{3,\infty} (\R^3 ))_{\rm loc}$ with  $\beta \in (0,1)$,
\end{itemize}
and at least one the two following properties holds true:
\begin{itemize}
\item $u$ is in $\widetilde{L}^{3} (0,T ; \dot{B}^{\alpha}_{3,c_{0}} (\R^3 ))_{\rm loc}$, 
\item $B$ is in $\widetilde{L}^{3} (0,T; \dot{B}^{\beta}_{3,c_{0}} (\R^3 ))_{\rm loc}$.
\end{itemize}
Finally assume that  $2 \alpha +  \beta \geq 1$ and $3 \beta \geq 1$.
Then  the local crossed  fluid-magneto-helicity identity  \eqref{localfm} holds true.
\end{enumerate}
\end{Theorem}
Therefore the Hall effect does not modify formally the laws 
of conservation of energy and magneto-helicity but it could be, 
in view of Theorem~\ref{Hallonsaghel}, 
Theorem~\ref{Hallonsagen} and Theorem~\ref{MHDonsag}, 
that it creates some extra anomalous dissipation in these laws 
for solutions which have only a quite bad regularity.

Observe that Theorem \ref{MHDonsag} extends to the case of the viscous resistive MHD some earlier results by \cite{CKS} about the ideal MHD.

\begin{Remark}
We do not reproduce here the dimensional argument 
(given after Theorem~\ref{LLMonsag}) for the HMHD and MHD equations 
but one can check that the conditions given in Theorem~\ref{Hallonsaghel}, 
Theorem~\ref{Hallonsagen} and Theorem~\ref{MHDonsag} are critical 
in the sense of this dimensional analysis.
\end{Remark}
The proof of Part i) and Part ii) of Theorem~\ref{MHDonsag} 
is left to the reader since it can be proved 
along the same lines as the proofs of Theorem~\ref{Hallonsaghel} 
and Theorem \ref{Hallonsagen}. The proof of the Part iii) is 
tackled in Section \ref{pom}.

\subsection{Suitable weak solutions}

Observe that despite we use the word ``dissipation'' in the statements 
above we do not claim anything about the sign of $e^\mathfrak{a}_{\rm MLL} $ 
or $e^\mathfrak{a}_{\rm HMHD} $.
This terminology would be particularly appropriate if the distributions $e^\mathfrak{a}_{\rm MLL} $ 
or $e^\mathfrak{a}_{\rm HMHD}$ were non positive.

 The corresponding feature for the Navier-Stokes equations has been quite useful in order to obtain partial regularity theorems limiting the parabolic
Hausdorff dimension of the singular set, see  \cite{CKN}.
In particular, in this context, the term  ``suitable"  has been coined for weak solutions that have a non positive anomalous dissipation.
Strikingly enough the approximation process used by Leray in order to establish 
the existence of weak solutions  actually leads to suitable weak solutions, 
see \cite{raoul}. 
In Leray's scheme, the approximate equations read:
\begin{eqnarray}
\label{suit1}
\dt u^{\eps} +   (u^{\eps})_{\eps} \cdot \nabla u^{\eps} + \nabla p^{\eps} &=&  \Delta u^{\eps} ,
\\ \label{suit2}
\Div u^{\eps} &=& 0 .
\end{eqnarray}
One may also argue that an appropriate sign condition on the anomalous dissipation 
could be helpful to select 
among weak solutions, which ones may be considered physically acceptable, 
as one might 
think that the lack of smoothness could lead to local energy creation.
Indeed in the case of the inviscid Burgers equation in one space dimension the requirement to be suitable coincides
with the usual entropy condition of negative jumps, which does imply uniqueness.
However such a result has not been proved yet for the Navier-Stokes equations, 
up to our knowledge, and the case of the HMHD equations could be even more difficult. 
In the case of the MLL equations, such a result is even more desirable since Alouges 
and Soyeur have proved in \cite{AS} non-uniqueness of weak solutions to the  
Landau-Lifshitz equation.

In this  section, we investigate what can be said about the sign of  
the anomalous energy dissipations $d^\mathfrak{a}_{\rm MLL} $ and  
$d^\mathfrak{a}_{\rm HMHD} $ for some rather standard processes used in order to 
prove the existence of weak solutions to  the MLL equations and of the HMHD equations.

\paragraph{Case of the MLL equations.}

One difficulty in establishing the existence of weak solutions to the MLL equations as claimed in Theorem \ref{CF} is due to the condition 
$|m| = 1$ almost everywhere. 
A brutal application of the usual strategy of mollification of the equation fails to capture this constraint. 
A by-now usual way to overcome this difficulty consists in using a Ginzburg-Landau type penalization, following the analysis performed in  \cite{AS} for the Landau-Lifshitz equation and  \cite{CF98} for the MLL equations. Here we will consider, for $\eps\in(0,1)$,  the penalized equations:
\begin{eqnarray}
\label{LLM-suit1}
\dt m^{\eps} -  m^{\eps} \wedge \dt m^{\eps} &=&
2 \Big( \Delta m^{\eps} + H^{\eps}- (H^{\eps} \cdot m^{\eps} ) m^{\eps} - \frac{1}{\eps} (| m^{\eps}|^2 -1) m^{\eps}   \Big) ,
\\ \label{LLM-suit2}
\dt H^{\eps} + \curl  E^{\eps} &=& - \dt m^{\eps} ,
\\ \label{LLM-suit3}
\dt E^{\eps} - \curl H^{\eps} &=& 0 , 
\\ \label{LLM-suit4}
\Div  E^{\eps} &=& \Div (H^{\eps}+m^{\eps}) = 0 .
\end{eqnarray}
Let us emphasize in particular that equation \eqref{LLM-suit1} is slightly different from the penalized equation used in \cite{CF98}. The difference is that we add the term $(H^{\eps} \cdot m^{\eps} ) m^{\eps}$ in order to be able to apply the weak maximum principle and to get a better regularity. A similar idea was used in \cite{DLW} 
and in \cite{kevin} for the quasi-stationary Landau-Lifshitz equations. 

As a  first step in order to prove Theorem \ref{CF}, one then establishes the existence 
of weak solutions to \eqref{LLM-suit1}-\eqref{LLM-suit4}, an easy task since the  
condition $|m| = 1$ a.e. has been dropped out. 
More precisely, one obtains that for initial data as in Theorem \ref{CF}, that is 
for  $m_0$ in $L^{\infty} (\R^3 ; \R^3)$ such that 
$|m_0| = 1$ almost everywhere and $\nabla m_0$ is in $L^{2} (\R^3 ; \R^{9})$, 
$E_{0}$ and $H_{0}$ in $L^{2} (\R^3 ; \R^3)$ such that 
$\Div  E_0 = \Div (H_0 + m_{0} )= 0 $, there  exists, for all $\eps\in(0,1)$,  a weak solution $(m^{\eps},E^{\eps},H^{\eps}) : (0,\infty)\times\R^3 \rightarrow \R^9$ 
of \eqref{LLM-suit1}-\eqref{LLM-suit4} on $(0,\infty)\times\R^3$, with initial value $(m_0,E_0,H_0)$, 
such that, for all $T>0$, 
\begin{gather*}
m^{\eps}  \in  L^\infty ((0,T) ; \dot{H}^{1} (\R^3 ; \R^3))  , \quad 
(E^{\eps}, H^{\eps})     \in 
L^\infty ((0,T) ; L^{2} (\R^3 ; \R^6))   ,  \\  
( |m^{\eps} |^2 - 1  , \nabla m^{\eps} ) \in L^\infty ((0,T); L^{2} (\R^3 ;  \R^{10} )) \quad \text{and} \quad  
\dt m^{\eps}  \in L^{2} ((0,T) \times \R^3 ;  \R^3)  ,
\end{gather*}
and 
\begin{equation}
\label{MLLenergyGL}
\forall T\geq0, \quad 
\Big( \mathcal{E}^{\eps}_{\rm MLL} (T) 
 + \mathcal{E}^{\eps}_{\rm GL} (T) 
+   \int_{0}^T   \mathcal{D}^{\eps}_{\rm MLL} (t)  \, dt  \Big)_\eps \text{ is  bounded uniformly} ,
\end{equation}
where
\begin{gather*}
\mathcal{E}^{\eps}_{\rm MLL} (t) := 
\int_{\R^3}  \Big(  | E^{\eps} |^{2} (t,x)  +  | H^{\eps} |^{2} (t,x) 
+  | \nabla m^{\eps} |^{2} (t,x) \Big) dx ,
\quad 
 \mathcal{D}^{\eps}_{\rm MLL} (t) :=   \int_{ \R^3} | \partial_{t} m^{\eps} |^{2} (t,x)  \, dx  ,
\\  \ \text{ and } \  \mathcal{E}^{\eps}_{\rm GL} (t) := 
\frac{1}{2\eps}  \int_{\R^3}  \Big(| m^{\eps} (t,x) |^2 -1  \Big)^2 dx .
\end{gather*}
Then, by standard compactness arguments, one infers that there exists 
\begin{gather*}
m  \in  L^\infty ((0,T) ; \dot{H}^{1} (\R^3 ; \R^3))     , \quad 
(E, H) \in L^\infty ((0,T) ; L^{2} (\R^3 ; \R^6))   ,  \\  
\text{with } ( |m |^2 - 1  , \nabla m ) \in L^\infty ((0,T); L^{2} (\R^3 ;  \R^{10} )) \quad \text{and} \quad  
\dt m  \in L^{2} ((0,T) \times \R^3 ;  \R^3)  ,
\end{gather*}
such that, up to a subsequence, $m^{\eps}$  weakly converges to $m$ in ${H}^{1} ((0,T) \times \R^3 ; \R^3)$,  $|m^{\eps} |^2 - 1 $ converges to $0$ weakly in $L^{2} ((0,T) \times \R^3 ; \R)$ and almost everywhere, $(E^{\eps}, H^{\eps}) $ converges to $(E,H)$  weakly in $L^{2} ((0,T) \times \R^3 ; \R^6)$.
Moreover $(m,E, H) $ is a weak solution as claimed in Theorem \ref{CF}. Since we do not claim here any originality let us simply refer for instance to  \cite{AS} where this step is detailed  for the Landau-Lifshitz equation, and to \cite{CF98} where the case of the MLL equations is tackled. 

Our point here is the following.
\begin{Theorem}
\label{THsuitableMLL}
Let $(m,E,H)$ be a weak solution to the MLL equations  obtained as a limit point of the sequence $(m^{\eps},E^{\eps},H^{\eps})$ as considered above. 
Assume moreover that, up to a subsequence,   $  H^{\eps} \wedge E^{\eps}   $  and, for $i=1,2,3$, $ \partial_{t} m^{\eps} \cdot   \partial_{i} m^{\eps} $ converge respectively to  $H  \wedge E$   and $ \partial_{t} m \cdot   \partial_{i} m $ in the sense of distributions, and that $H^{\eps} \cdot m^{\eps} $  converges in $L^2_{\rm loc} ((0,T) \times (\R^3 ))$ to $H \cdot m$.
Then there exist two non negative distributions $d^{\mathfrak{a},1}_{\rm MLL} $ and $e^{\mathfrak{a}}_{\rm MLL} $ such that 
$d^\mathfrak{a}_{\rm MLL} = - d^{\mathfrak{a},1}_{\rm MLL}  - \partial_t e^{\mathfrak{a}}_{\rm MLL}$. 
\end{Theorem}
This theorem can be seen somehow as a counterpart of  \cite[Proposition 4]{raoul} 
which shows that any weak solution to the Euler equation which is a strong limit 
of  suitable  solutions to the Navier-Stokes equation, 
as viscosity goes to zero, is a suitable weak solution. 
Let us stress that in Theorem \ref{THsuitableMLL} the assumption is weaker than the strong convergence of $H^{\eps}$, $ E^{\eps}   $, $H^{\eps} \cdot m^{\eps} $ and $\nabla m^{\eps} $ in $L^2_{\rm loc} ((0,T) \times (\R^3 ))$.
  On the other hand when this strong convergence holds, the proof will reveal that $d^{\mathfrak{a}}_{\rm MLL}$ vanishes and that $e^{\mathfrak{a}}_{\rm MLL}$ is only due to the possible lack of strong convergence of the energy density 
  \begin{equation}
\label{88}
 e^{\eps}_{\rm GL} := 
\frac{1}{2\eps}   \big(| m^{\eps}  |^2 -1  \big)^2
\end{equation}
   associated with $\mathcal{E}^{\eps}_{\rm GL}$. It would be interesting to investigate the existence of another way to construct weak solutions 
  to the MLL equations for which the distribution $e^{\mathfrak{a}}_{\rm MLL}$ vanishes as well.

\paragraph{Case of the HMHD equations.}

Mimicking the mollification process in \eqref{suit1}-\eqref{suit2}, we consider the equations:
\begin{eqnarray}
\label{HMHD-suit1}
\dt u^{\eps} +   (u^{\eps})_{\eps} \cdot \nabla u^{\eps} + \nabla p^{\eps} &=& (\curl B^{\eps}) \wedge (B^{\eps})_{\eps} + \Delta u^{\eps}
\\ \label{HMHD-suit2}
\Div u^{\eps} &=& 0 ,
\\ \label{HMHD-suit3}
\dt B^{\eps} - \curl  (u^{\eps}  \wedge  ( B^{\eps})_{\eps}  ) + \curl  \Big(  ( \curl B^{\eps} )  \wedge ( B^{\eps})_{\eps} \Big) 
&=& \Delta B^{\eps} ,
\\ \label{HMHD-suit4}
\Div  B^{\eps} &=& 0 .
\end{eqnarray}
Standard arguments yield that for all $u_0$ and $B_0$  in $\mathcal{H}$, for all $\eps\in(0,1)$,  there exists a global weak solution  
$$(u^{\eps} , B^{\eps} ) \in 
\Big( L^\infty (0,T ; \mathcal{H} )\cap L^2 (0,T ; \mathcal{V} ) \Big)^2$$
solution to 
\eqref{HMHD-suit1}-\eqref{HMHD-suit4} corresponding to these initial data. 
Moreover
this solution satisfies the  energy inequality \eqref{nrjHALL}.
Therefore, up to a subsequence, $u^{\eps}$ and $ B^{\eps}$ converge in $L^\infty (0,T ; \mathcal{H} )$ weak-* 
and weakly  in $ L^2 (0,T ; \mathcal{V} ) $ respectively to $u$ and $B$.  
Using some a priori  temporal estimates and Aubin-Lions' lemma, we deduce that,  up to a subsequence, $u^{\eps}$ and $ B^{\eps}$ converges in $L^{3} ((0,T) \times \R^{3} )_{\rm loc}$.
This allows to establish that $(u,B)$ is a weak solution to the HMHD equations associated with the initial data $(u_0 ,B_0 )$. 
We refer to \cite{HMHD,CDL} for the details of this procedure, 
though for a slightly different regularization scheme.

Our point here is the following.
\begin{Theorem}
\label{THsuitableHMHD}
Let $(u,B)$ be a weak solution to the HMHD equations obtained as a limit point of 
the sequence $(u^\eps,B^\eps)$ as considered above.  
Assume moreover that, up to a subsequence, $B^{\eps}$ converges to $B$ in $L^{4} ((0,T) \times \R^{3} )_{\rm loc}$. 
 Then the anomalous energy dissipation $d^{\mathfrak{a}}_{\rm HMHD}$ is non positive.
 \end{Theorem}
Let us mention that the assumption that, up to a subsequence, $B^{\eps}$ converges to $B$ in $L^{4} ((0,T) \times \R^{3} )_{\rm loc}$, can be dropped out in the case of the MHD equations, as 
will be shown in the proof.

\subsection{A few extra comments}

An analysis relying on Littlewood-Paley decomposition as in \cite{CCFS} 
for the incompressible Euler equations can be transposed to the MLL and HMHD systems. 
We will not go in this direction here. 
Let us only mention that it seems that the same range of regularity is 
attained by both methods.

\par \ \par

In \cite{Chemin}, Chemin recently used a strategy which is precisely close to the Littlewood-Paley counterpart of 
the key lemma used in \cite{CET}, and adapted below, cf. Section \ref{CETtype}, in order to obtain some sharp weak-strong uniqueness 
result for the incompressible Navier-Stokes equations, see in particular 
Lemma 2.3 in \cite{Chemin}. 

It could therefore be interesting to see if it is possible to sharpen the weak-strong uniqueness 
results given in Theorem
\ref{WSMLL} and  in Theorem \ref{WSHMHD}, that is to extend the statements to rougher strong solutions, using for example the regularization approach.

\par \ \par 

Since the results obtained here are local in space, they could easily 
be adapted to the case where the equations are set in a bounded domain; 
one then obtains a sufficient condition for the dissipation to 
vanish in the interior of the domain.
However the counterpart of such a result up to the boundary seems more 
difficult. 
In this direction let us mention the results \cite{CFS,FT} about the 
incompressible Navier-Stokes equations. 
Yet, with respect to this issue the MLL and HMHD equations seem to be closer 
to the incompressible Euler equations, for which there is, to our knowledge, 
no result of regularity up to the boundary.

%%%%%%%%%%%%%
%%%%%%%%%%%%%
%%%%%%%%%%%%%
%%%%%%%%%%%%%

\section{Weak-strong uniqueness: Proof of Theorem \ref{WSLL} and of Theorem \ref{WSMLL}}
\label{proofWSMLL}

\subsection{Case of the Landau-Lifshitz equation}

For sake of expository, we shall first give  the proof  of Theorem \ref{WSLL} which deals with 
the case of the Landau-Lifshitz equation \eqref{heu}, where the fields $E$ and $H$ are omitted. 
The  extension to the MLL equations \eqref{LLGC}-\eqref{LLMGC2}-\eqref{LLMGC3}-\eqref{LLMGC4} is given in the next subsection. 

We therefore consider $m_{2}: (0,\infty)\times\R^3 \rightarrow \R^3$ a global weak solution  to \eqref{heu} on $(0,\infty)\times\R^3$ satisfying the energy inequality \eqref{LLenergy}, for almost every $T\geq0$.
Here the weak formulation of  the Landau-Lifshitz equation \eqref{heu} reads: for every $\Psi \in H^{1} ( (0,\infty)\times\R^3 ; \R^3 )$, 
\begin{equation}
\label{wfLL}
\int_{0}^T  \int_{\R^3} ( \partial_{t} m_{2} + m_{2} \wedge \partial_{t} m_{2} ) \cdot \Psi   \, dx \, dt 
  = - 2  \sum_i \int_{0}^T  \int_{\R^3} (m_{2} \wedge \partial_{i} m_{2} ) \cdot  \partial_{i} \Psi   \, dx \, dt ,
\end{equation}
 where the sum is over $1,2,3$.
 The initial data $m_{0}$ is prescribed in the trace sense, and is here assumed to be smooth.

Let us consider $ m_{1}$, a smooth solution to  \eqref{heu} with the same initial data $m_{0}$.
We denote $m := m_{1} - m_{2}$ and expand $J_{\rm LL} [ m](T) $ into 
\begin{equation*}
J_{\rm LL} [ m](T) = J_{\rm LL} [ m_{1}](T) + J_{\rm LL} [ m_{2}](T)
- 2  \big(  \int_{\R^3}  \nabla m_{1} :   \nabla m_{2}   \, dx \big) (T)  
-2  \int_{0}^T  \int_{\R^3} \partial_{t} m_{1} \cdot \partial_{t} m_{2}  \, dx \, dt  .
\end{equation*}
Using some integration by parts, we have 
\begin{eqnarray} 
\label{tolaeb}
\big(  \int_{\R^3}  \nabla m_{1} :   \nabla m_{2}   \, dx \big) (T)  
&=&  \sum_i \int_{0}^T  \int_{\R^3}  ( \partial_{i}  \partial_{t}  m_{1} ) \cdot  \partial_{i}   m_{2}  \, dx \, dt
\\ \nonumber &&-  \int_{0}^T  \int_{\R^3}  (\Delta m_{1} ) \cdot \partial_{t}   m_{2}    \, dx \, dt
+  \int_{\R^3}    | \nabla m_{0} |^{2}  \, dx  .
\end{eqnarray}
Now, the two solutions satisfy the  energy inequality \eqref{LLenergy}, so that, for almost every $T\geq0$,
\begin{equation}
\label{Ple}
J_{\rm LL} [ m](T)  \leqslant K_{\rm LL} [ m_{1} ,m_{2}](T) ,
\end{equation}
where
\begin{eqnarray}
\nonumber
 K_{\rm LL} [ m_{1} ,m_{2}](T) &:=& 
 -2 \sum_i \int_{0}^T  \int_{\R^3}  ( \partial_{i}  \partial_{t}  m_{1} ) \cdot  \partial_{i}   m_{2}  \, dx \, dt
+2  \int_{0}^T  \int_{\R^3}  (\Delta m_{1} ) \cdot \partial_{t}   m_{2}    \, dx \, dt
\\  \label{Pledge}
&&-2  \int_{0}^T  \int_{\R^3} \partial_{t} m_{1} \cdot \partial_{t} m_{2}  \, dx \, dt  .
\end{eqnarray}
We now use  the weak formulation \eqref{wfLL} with the test functions $\Psi =  \partial_{t}  m_{1}$ and $\Psi =-2 \Delta m_{1}$:
\begin{eqnarray}
\label{wfa}
  \int_{0}^T  \int_{\R^3} ( \partial_{t} m_{2} + m_{2} \wedge \partial_{t} m_{2} ) \cdot    \partial_{t}  m_{1}  \, dx \, dt 
  = - 2  \sum_i \int_{0}^T  \int_{\R^3} (m_{2} \wedge \partial_{i} m_{2} ) \cdot  \partial_{i}  \partial_{t}  m_{1}   \, dx \, dt ,
  \\ \label{wfb}
  - 2 \int_{0}^T  \int_{\R^3} ( \partial_{t} m_{2} + m_{2} \wedge \partial_{t} m_{2} ) \cdot  \Delta m_{1}  \, dx \, dt 
  = 4  \sum_i \int_{0}^T  \int_{\R^3} (m_{2} \wedge \partial_{i} m_{2} ) \cdot  \partial_{i}  \Delta m_{1}   \, dx \, dt .
\end{eqnarray}
On the other hand, since $m_{1}$ is a strong solution to \eqref{heu}, we have
\begin{eqnarray}
\label{stro1}
  \int_{0}^T  \int_{\R^3} ( \partial_{t} m_{1} + m_{1} \wedge \partial_{t} m_{1} ) \cdot    \partial_{t}  m_{2}  \, dx \, dt 
  =  2  \sum_i \int_{0}^T  \int_{\R^3} (m_{1} \wedge \Delta m_{1} ) \cdot   \partial_{t}  m_{2}   \, dx \, dt ,
\\ \label{stro2}
 2 \int_{0}^T  \int_{\R^3}\partial_{i}   ( \partial_{t} m_{1} + m_{1} \wedge \partial_{t} m_{1} ) \cdot    \partial_{i}  m_{2}  \, dx \, dt 
  = 4  \sum_i \int_{0}^T  \int_{\R^3}\partial_{i}   (m_{1} \wedge \partial_{i} m_{1} ) \cdot  \partial_{i}  m_{2}   \, dx \, dt .
\end{eqnarray}
Thanks to \eqref{wfa}-\eqref{stro2} we get that 
\begin{equation*}
 K_{\rm LL} [ m_{1} ,m_{2}]  := -4 (I_{1} + I_{2} ) + 2 (I_{3} + \ldots + I_{6} ) + (I_{7} + I_{8} ) , 
\end{equation*}
with
\begin{eqnarray*}
I_{1} :=    \sum_i   \int_{0}^T  \int_{\R^3} (   m_{2} \wedge   \partial_{i} m_2 ) \cdot  \partial_{i}     \Delta m_{1}   \, dx \, dt
                           , \quad I_{2} :=  \sum_i   \int_{0}^T  \int_{\R^3}  \big(  \partial_{i}   (   m_{1} \wedge \Delta m_{1} )  \big) \cdot    \partial_{i} m_2          \, dx \, dt              ,
\\ I_{3} :=   \sum_i   \int_{0}^T  \int_{\R^3}   \partial_{i}   (   m_{1} \wedge    \partial_{t} m_{1} )  \cdot    \partial_{i} m_2   \, dx \, dt
                              , \quad I_{4} := -  \int_{0}^T  \int_{\R^3}    (   m_{2} \wedge    \partial_{t} m_{2} )  \cdot \Delta m_{1}  \, dx \, dt ,
\\ I_{5} :=        -        \int_{0}^T  \int_{\R^3} (   m_{1} \wedge \Delta m_{1} ) \cdot   \partial_{t} m_{2} \, dx \, dt
              , \quad I_{6} :=   \sum_i  \int_{0}^T  \int_{\R^3} (   m_{2} \wedge    \partial_{i} m_{2} )  \cdot  \partial_{i}   \partial_{t} m_1  \, dx \, dt,
\\ I_{7} :=      \int_{0}^T  \int_{\R^3} (   m_{2} \wedge    \partial_{t} m_{2} )  \cdot \partial_{t} m_1            \, dx \, dt
                       , \quad I_{8} := \int_{0}^T  \int_{\R^3} (   m_{1} \wedge    \partial_{t} m_{1} )  \cdot \partial_{t} m_2  \, dx \, dt.   
\end{eqnarray*}

Using Leibniz' rule and the properties of the triple product, we get 
\begin{eqnarray*}
I_{1} + I_{2} &=& 
-   \sum_i  \int_{0}^T  \int_{\R^3} (   m_{2}  \wedge     \partial_{i}  \Delta m_{1}    ) \cdot  \partial_{i}   m_{2}  \, dx \, dt
+  \sum_i  \int_{0}^T  \int_{\R^3} (   m_{1}  \wedge     \partial_{i}  \Delta m_{1}    ) \cdot  \partial_{i}   m_{2}  \, dx \, dt
\\ &&+ \sum_i  \int_{0}^T  \int_{\R^3} (    \partial_{i}   m_{1}  \wedge   \Delta m_{1}    ) \cdot  \partial_{i}   m_{2}  \, dx \, dt
\\  &=&   \sum_i  \int_{0}^T  \int_{\R^3} (   m  \wedge     \partial_{i}  \Delta m_{1}    ) \cdot  \partial_{i}   m_{2}  \, dx \, dt
- \sum_i  \int_{0}^T  \int_{\R^3} (    \partial_{i}   m_{1}  \wedge   \Delta m_{1}    ) \cdot  \partial_{i}   m  \, dx \, dt
\\  &=&   \sum_i  \int_{0}^T  \int_{\R^3} (   m  \wedge     \partial_{i}  \Delta m_{1}    ) \cdot  \partial_{i}   m_{2}  \, dx \, dt 
- \sum_i  \int_{0}^T  \int_{\R^3} (    \partial_{i}   m  \wedge   \Delta m_{1}    ) \cdot  \partial_{i}   m  \, dx \, dt 
\\   &=& -   \sum_i  \int_{0}^T  \int_{\R^3} (   m  \wedge     \partial_{i}  \Delta m_{1}    ) \cdot  \partial_{i}   m_{2}    \, dx \, dt 
+   \sum_i   \int_{0}^T  \int_{\R^3}\big( \partial_{i} ( m  \wedge     \Delta m_{1}    )\big) \cdot  \partial_{i}   m_{1}    \, dx \, dt 
\\ &&-   \sum_i  \int_{0}^T  \int_{\R^3} ( m  \wedge     \partial_{i}  \Delta m_{1}    ) \cdot  \partial_{i}   m_{1}  \, dx \, dt .
\end{eqnarray*}
By an integration by parts, we see that the second term in the right hand side above vanishes. The remaining two terms can be combined into
\begin{equation}
\label{12}
I_{1} + I_{2} = -  \sum_i   \int_{0}^T  \int_{\R^3} (   m  \wedge     \partial_{i}  \Delta m_{1}    ) \cdot  \partial_{i}   m   \, dx \, dt.
\end{equation}

We use again Leibniz' rule to expand $I_{3}$ into 
\begin{equation}
\label{splittin}
I_{3} =   \sum_i   \int_{0}^T  \int_{\R^3}   \big((  \partial_{i}    m_{1}) \wedge    \partial_{t} m_{1} \big)  \cdot    \partial_{i} m_2    \, dx \, dt
+   \sum_i   \int_{0}^T  \int_{\R^3}   (      m_{1} \wedge  \partial_{i}  \partial_{t} m_{1} )  \cdot    \partial_{i} m_2    \, dx \, dt
=:  I_{3,a}  + I_{3,b}  .
\end{equation}
Then, we observe that
\begin{eqnarray}
 I_{3,b}  + I_{6} &=&  \sum_i  \int_{0}^T  \int_{\R^3} \det (m , \partial_i \partial_{t} m_{1} , \partial_i m_{2} )  \, dx \, dt ,
 \\  I_{4}  + I_{5} &=& \int_{0}^T  \int_{\R^3}  \det (m , \partial_{t} m_{2} , \Delta m_{1}  )  \, dx \, dt .
\end{eqnarray}
On the other hand, we have 
\begin{eqnarray}
\nonumber
 I_{3,a}  &=&  -  \sum_i   \int_{0}^T  \int_{\R^3} \big(  (  \partial_{i}    m_{1}) \wedge    \partial_{t} m_{1} \big)  \cdot    \partial_{i} m   \, dx \, dt
  \\  \label{I3a} &=& \int_{0}^T  \int_{\R^3}   \big(  \Delta m_{1} \wedge    \partial_{t} m_{1} )  \cdot  m   \, dx \, dt
  +    \sum_i   \int_{0}^T  \int_{\R^3} \big(  (  \partial_{i}    m_{1}) \wedge    \partial_{t} m_{1} \big)  \cdot   m   \, dx \, dt ,
\end{eqnarray}
by integration by parts.

Combining \eqref{splittin}-\eqref{I3a} we obtain
\begin{equation}
\label{36}
I_{3} + \ldots + I_{6} = -  \sum_i   \int_{0}^T  \int_{\R^3} \det (m,  \partial_{t} \partial_{i}    m_{1} , \partial_{i}    m )   \, dx \, dt
-   \int_{0}^T  \int_{\R^3}   \det (m,  \partial_{t} m,  \Delta m_{1} )  \, dx \, dt .
\end{equation}

We have 
\begin{eqnarray}
\nonumber
I_{7} +  I_{8} &=&  \int_{0}^T  \int_{\R^3} \Big(   \det (m,  \partial_{t} m_{2} ,  \partial_{t} m_{1} ) +  \det (\partial_{t} m_{2} ,  m_{1} ,  \partial_{t} m_{1} ) 
\Big)   \, dx \, dt
\\  \label{78} &=& -   \int_{0}^T  \int_{\R^3}   \det (m,  \partial_{t} m_{2} ,  \partial_{t} m_{1} )   \, dx \, dt
=   \int_{0}^T  \int_{\R^3}   \det (m,  \partial_{t} m ,  \partial_{t} m_{1} )  \, dx \, dt .
\end{eqnarray}

Now, summing \eqref{12}, \eqref{36} and  \eqref{78}, we deduce that 
\begin{eqnarray}
\label{distingo}
 K_{\rm LL} [ m_{1} ,m_{2}]  &=& 4 \sum_i \int_{0}^T  \int_{\R^3}  (m \wedge  \partial_{i}  \Delta   m_{1} ) \cdot  \partial_{i} m   \, dx \, dt
 -  2\sum_i \int_{0}^T  \int_{\R^3}  (m \wedge  \partial_{t}  \partial_{i} m_{1} ) \cdot   \partial_{i} m  \, dx \, dt
\\  \nonumber && - 2 \int_{0}^T  \int_{\R^3}  (m \wedge  \partial_{t} m) \cdot  \Delta   m_{1}  \, dx \, dt
 +  \int_{0}^T  \int_{\R^3}  (m \wedge  \partial_{t} m) \cdot  \partial_{t} m_{1}  \, dx \, dt.
\end{eqnarray}

Since $m$ vanishes at initial time, Poincar\'e's inequality yields 
$$
 \int_{0}^T  \int_{\R^3} | m |^{2} \leqslant o(T)  \int_{0}^T  \int_{\R^3} | \partial_{t} m |^{2} \, dx \, dt .
 $$
Thus, for $T$ small enough, 
one gets 
\begin{eqnarray*}
 |K_{\rm LL} [ m_{1} ,m_{2}]  | \leqslant \frac12  \int_{0}^T  \int_{\R^3} | \partial_{t} m |^{2} \, dx \, dt
 + C  \int_{0}^T  \int_{\R^3} | \nabla m |^{2}  \, dx \, dt.
 \end{eqnarray*}
We combine this with inequality 
\eqref{Ple} and use Gronwall's lemma to conclude that $m$ vanishes, first for small time, and the argument can be repeated as many times as necessary.
%

%%%%%%%%%%%%%
%%%%%%%%%%%%%
%%%%%%%%%%%%%
%%%%%%%%%%%%%

\subsection{Case of the MLL equations}

We  consider $(m_{2} ,E_2,H_2) : (0,\infty)\times\R^3 \rightarrow \R^9$ a global weak solution  to the MLL equations  on $(0,\infty)\times\R^3$ satisfying the energy inequality, for almost every $T\geq0$.
Here the weak formulation reads: for every $\Psi \in H^{1} ( (0,T)\times\R^3 ; \R^3 )$, 
\begin{gather}
\label{wf}
  \int_{0}^T  \int_{\R^3} ( \partial_{t} m_{2} + m_{2} \wedge \partial_{t} m_{2} ) \cdot \Psi   \, dx \, dt 
  = - 2  \sum_i \int_{0}^T  \int_{\R^3} (m_{2} \wedge \partial_{i} m_{2} ) \cdot  \partial_{i} \Psi   \, dx \, dt
\\ \nonumber   + 2 \int_{0}^T  \int_{\R^3} (m_{2} \wedge H_2 ) \cdot \Psi  \, dx \, dt
   ,
   \\ \label{wf2}
   - \int_{0}^T     \int_{\R^3} (H_{2}+m_{2}) \cdot \dt  \Psi   \, dx \, dt + \int_{0}^T     \int_{\R^3} E_2  \cdot \curl  \Psi   \, dx \, dt 
  =  \int_{\R^3} (H_0 +m_0 ) \cdot \Psi (0,\cdot) \, dx
  \\ \nonumber   -  (  \int_{\R^3} (H_{2}+m_{2}) \cdot \Psi \, dx )(T)  ,
  \\ \label{wf3}
  -\int_{0}^T     \int_{\R^3}   E_2 \cdot  \partial_{t}  \Psi    \, dx \, dt  
  - \int_{0}^T     \int_{\R^3} H_{2}  \cdot \curl  \Psi  \, dx \, dt
  = \int_{\R^3} E_0 \cdot \Psi (0,\cdot) \, dx
  - (  \int_{\R^3} E_{2}\cdot  \Psi \, dx )(T)    ,
 \end{gather}
 where the sum is over $1,2,3$.
 The initial data $m_{0}$ is prescribed in the trace sense, and is here assumed to be smooth.

Let us consider $ (m_{1}  ,E_1 ,H_1)$, regular solution to the MLL equations with the same initial data $m_{0}$.

For $j=1,2$, let
\begin{gather*}
\mathcal{E}^j_{\rm MLL} (t) := 
\int_{\R^3}  \Big(  | E_j |^{2} (t,x)  +  | H_j |^{2} (t,x) 
+  | \nabla m_j |^{2} (t,x) \Big) dx 
 \ \text{ and } \ 
 \mathcal{D}^j_{\rm MLL} (t) :=   \int_{ \R^3} | \partial_{t} m_j |^{2} (t,x)  \, dx   .
\end{gather*}
Let us also introduce
\begin{equation*}
J^j_{\rm MLL}  (T) :=   \mathcal{E}^j_{\rm MLL}(T)  
 +  \int_{0}^T \mathcal{D}^j_{\rm MLL}   \, dt  
\end{equation*}
and
\begin{equation*}
L_{\rm MLL}   (T) :=  \int_{\R^3}  \Big(  | E |^{2} (T,x)  +  | H |^{2} (T,x) 
+  | \nabla m |^{2} (T,x) \Big) dx  +\int_{0}^T    \int_{ \R^3} | \partial_{t} m |^{2} (t,x)  \, dx  \, dt  ,
\end{equation*}
where $m := m_{1} - m_{2}$, $E:=E_1 -E_2$ and $H:= H_1 - H_2$.

We first expand $L_{\rm MLL} (T) $ into 
\begin{eqnarray*}
L_{\rm MLL} (T) &=& J^1_{\rm MLL}(T) + J^2_{\rm MLL} (T)
- 2  \big(  \int_{\R^3}  \nabla m_{1} :   \nabla m_{2}   \, dx \big) (T)  -2  \int_{0}^T  \int_{\R^3} \partial_{t} m_{1} \cdot \partial_{t} m_{2}  \, dx \, dt 
\\ &&- 2 \big(  \int_{\R^3} E_1 \cdot E_2 \, dx \big) (T)  
- 2 \big(  \int_{\R^3} H_1 \cdot H_2 \, dx \big) (T)    .
\end{eqnarray*}
Since the two solutions satisfy the energy inequality, and using \eqref{tolaeb}, we get
\begin{eqnarray}
\label{pourgro}
L_{\rm MLL} (T) \leqslant 2 \tilde{L}_{\rm MLL} (T)  +  \tilde{K}_{\rm MLL} (T) ,
\end{eqnarray}
with
\begin{eqnarray*}
 \tilde{K}_{\rm MLL} (T) &:=& - 2 \sum_i \int_{0}^T  \int_{\R^3}  ( \partial_{i}  \partial_{t}  m_{1} ) \cdot  \partial_{i}   m_{2}  \, dx \, dt
 +2 \sum_i   \int_{0}^T  \int_{\R^3}  (\Delta m_{1} ) \cdot \partial_{t}   m_{2}    \, dx \, dt
\\ &&
-2  \int_{0}^T  \int_{\R^3} \partial_{t} m_{1} \cdot \partial_{t} m_{2}  \, dx \, dt   ,
\end{eqnarray*}
and
$$
\tilde{L}_{\rm MLL} (T) := 
   \int_{\R^3}  \Big(  | E_0 |^{2} (T,x)  +  | H_0 |^{2} (T,x)  \Big) dx 
-  \big(  \int_{\R^3} E_1 \cdot E_2 \, dx \big) (T)  
-  \big(  \int_{\R^3} H_1 \cdot H_2 \, dx \big) (T)  
.$$

Following the computations performed for the LL equations, taking into account the extra-term coming from the magnetic field in \eqref{wf}, we obtain:
\begin{eqnarray*}
 \tilde{K}_{\rm MLL} (T) &:=&
K_{\rm LL} [m_1 ,m_2 ] (T) 
- 2  \int_{0}^T  \int_{\R^3} \det (m_2 ,H_2 , \dt m_1 )  \, dx \, dt
+ 4  \int_{0}^T  \int_{\R^3} \det (m_2 ,H_2 , \Delta m_1 )  \, dx \, dt
\\ &&- 2 \int_{0}^T  \int_{\R^3}  \det (m_1 , H_1 , \dt m_2 )  \, dx \, dt
+ 4 \sum_i  \int_{0}^T  \int_{\R^3} \big(  \partial_i ( m_1 \wedge H_1 ) \big) \cdot \partial_i  m_2  \, dx \, dt.
\end{eqnarray*}
where $K_{\rm LL} [m_1 ,m_2 ] (T) $ denotes here the right-hand-side of 
\eqref{distingo}.

We use  \eqref{wf2} and  \eqref{wf3} respectively with $\Psi = H_1$ and $\Psi = E_1$ to get
\begin{eqnarray*}
\tilde{L}_{\rm MLL} (T) &=&    - \int_{0}^T     \int_{\R^3} (H_{2}+m_{2}) \cdot \dt H_1  \, dx \, dt + \int_{0}^T     \int_{\R^3} E_2  \cdot \curl  H_1  \, dx \, dt
\\ && -    \int_{\R^3}  H_0 \cdot m_0  \, dx +  \int_{\R^3}  m_2  (T) \cdot H_1 (T)   \, dx
\\ &&- \int_{0}^T     \int_{\R^3} E_2  \cdot  \dt E_1  \, dx \, dt
- \int_{0}^T     \int_{\R^3} H_2  \cdot \curl  E_1  \, dx \, dt,
\end{eqnarray*}
and then, using that $ (m_{1}  ,E_1 ,H_1)$  satisfies Equations \eqref{LLMGC2} and \eqref{LLMGC3}, we obtain
\begin{eqnarray*}
\tilde{L}_{\rm MLL} (T) = \int_{0}^T     \int_{\R^3}  H_2  \cdot \dt m_1  \, dx \, dt +  \int_{0}^T     \int_{\R^3}  H_1  \cdot \dt m_2  \, dx \, dt .
\end{eqnarray*}
Now, we use on one hand that $m_{1}$ solves equation \eqref{LLMGC1} and on the other hand equation \eqref{wf} with $\Psi = H_1$ to obtain 
\begin{eqnarray*}
\tilde{L}_{\rm MLL} (T) &=&
- \int_{0}^T     \int_{\R^3}   \det ( H_2 , m_1 ,\dt m_1 ) \, dx \, dt 
+ 2 \int_{0}^T     \int_{\R^3}   \det ( H_2 , m_1 ,\Delta m_1 ) \, dx \, dt 
\\ &&+ 2  \int_{0}^T     \int_{\R^3}   \det ( H_2 , m_1 , H_1 ) \, dx \, dt 
\\ && - \int_{0}^T     \int_{\R^3}   \det ( H_1 , m_2 ,\dt m_2 ) \, dx \, dt 
- 2 \sum_i  \int_{0}^T     \int_{\R^3}   (\partial_i H_1 ) \cdot  (m_2 \wedge \partial_i m_2 ) \, dx \, dt 
\\ &&+ 2 \int_{0}^T     \int_{\R^3}  \det ( H_1 , m_2 , H_2 )  \, dx \, dt  .
\end{eqnarray*}
Therefore
\begin{eqnarray}
\label{djvu}
2 \tilde{L}_{\rm MLL} (T)  +  \tilde{K}_{\rm MLL} (T) 
= K_{\rm LL} [m_1 ,m_2 ] (T) +
2 P_1 (T) + 4 P_2 (T) + 4 P_3  (T) ,
\end{eqnarray}
where
\begin{eqnarray*}
 P_1 (T) &:=&-   \int_{0}^T  \int_{\R^3} \det (m_2 ,H_2 , \dt m_1 )  \, dx \, dt 
-  \int_{0}^T  \int_{\R^3}  \det (m_1 , H_1 , \dt m_2 )  \, dx \, dt 
 \\ &&  -\int_{0}^T     \int_{\R^3}   \det ( H_2 , m_1 ,\dt m_1 ) \, dx \, dt 
-  \int_{0}^T     \int_{\R^3}   \det ( H_1 , m_2 ,\dt m_2 ) \, dx \, dt ,
\end{eqnarray*}
\begin{eqnarray*}
 P_2 (T) &:=& \sum_i   \int_{0}^T  \int_{\R^3} \big(  \partial_i ( m_1 \wedge H_1 ) \big) \cdot \partial_i  m_2   \, dx \, dt 
+  \int_{0}^T     \int_{\R^3}   \det ( H_2 , m_1 ,\Delta m_1 ) \, dx \, dt 
\\ && -   \sum_i  \int_{0}^T     \int_{\R^3}   (\partial_i H_1 ) \cdot  (m_2 \wedge \partial_i m_2 ) \, dx \, dt 
+   \int_{0}^T  \int_{\R^3} \det (m_2 ,H_2 , \Delta m_1 )  \, dx \, dt ,
\end{eqnarray*}
and
\begin{eqnarray*}
P_3  (T)  :=  \int_{0}^T     \int_{\R^3}  \det ( H_1 , m_2 , H_2 )  \, dx \, dt  + \int_{0}^T     \int_{\R^3}   \det ( H_2 , m_1 , H_1 ) \, dx \, dt .
\end{eqnarray*}
Now, we observe that 
\begin{eqnarray}
\label{pP1}
 P_1 (T) = -   \int_{0}^T  \int_{\R^3} \det (m ,H , \dt m_1 )  \, dx \, dt 
+ \int_{0}^T     \int_{\R^3}   \det ( m , H_1 ,\dt m ) \, dx \, dt ,
\end{eqnarray}
and, by Leibniz' rule, that
\begin{eqnarray*}
 P_2 (T) &=&  
 \sum_i  \int_{0}^T  \int_{\R^3} \det ( \partial_i m_1 ,  H_1 , \partial_i  m_2 )   \, dx \, dt 
 + \sum_i   \int_{0}^T  \int_{\R^3}  \det ( m_1 , \partial_i  H_1 , \partial_i  m_2 )   \, dx \, dt 
\\ && +  \int_{0}^T     \int_{\R^3}   \det ( H_2 , m ,\Delta m_1 ) \, dx \, dt 
 - \sum_i   \int_{0}^T     \int_{\R^3}   \det  (\partial_i H_1 , m_2 , \partial_i m_2 ) \, dx \, dt .
 \end{eqnarray*}
Now, we use an integration by parts to get that
\begin{eqnarray*}
\int_{0}^T  \int_{\R^3} \det ( \partial_i m_1 ,  H_1 , \partial_i  m_2 )   \, dx \, dt  &=& \int_{0}^T  \int_{\R^3} \det ( \partial_i m_1 ,  H_1 , \partial_i  m )   \, dx \, dt 
\\ &=& - \int_{0}^T  \int_{\R^3} \det ( \partial_i^2 m_1 ,  H_1 ,  m )   \, dx \, dt 
\\ &&-  \int_{0}^T  \int_{\R^3} \det ( \partial_i m_1 , \partial_i H_1 ,   m )   \, dx \, dt  .
 \end{eqnarray*}
Thus
\begin{eqnarray}
\label{pP2}
 P_2 (T) =  - \int_{0}^T  \int_{\R^3} \det ( \Delta m_1 ,  H ,  m )   \, dx \, dt 
 -  \sum_i \int_{0}^T  \int_{\R^3} \det ( \partial_i m , \partial_i H_1 ,   m )   \, dx \, dt  .
\end{eqnarray}
Finally, we easily get
\begin{eqnarray}
\label{pP3}
 P_3 (T) =  \int_{0}^T     \int_{\R^3}  \det ( H_1 , m , H )  \, dx \, dt   .
 \end{eqnarray}
Plugging  \eqref{djvu}-\eqref{pP1}-\eqref{pP2}-\eqref{pP3}
into  \eqref{pourgro} one gets
\begin{eqnarray*}
L_{\rm MLL} (T)  
 \leqslant \frac12  \int_{0}^T  \int_{\R^3} | \partial_{t} m |^{2} \, dx \,  dt
 + C  \int_{0}^T  \int_{\R^3}  \Big(  | E |^{2}  +  | H |^{2} 
+  | \nabla m |^{2} \Big)  \, dx \,  dt,
 \end{eqnarray*}
and using a Gronwall lemma yields the desired conclusion.

%%%%%%%%%%%%%
%%%%%%%%%%%%%
%%%%%%%%%%%%%
%%%%%%%%%%%%%

\section{Weak-strong uniqueness: Proof of Theorem \ref{WSHMHD}}
\label{proofWSHMHD}

We will prove Theorem  \ref{WSHMHD} in a simplified setting which focuses on the difficulty due to the Hall effect. The extension to the general case is straightforward. 
Therefore we consider the following equations:
\begin{gather}
\label{HallFort1}
\partial_{t} B + \curl \big( (\curl B \wedge B) \big) = \Delta B ,
\\ \label{HallFort2} \Div B =0 .
\end{gather}
We consider a global weak solution $B_{2}$ to \eqref{HallFort1}-\eqref{HallFort2} associated with an initial data $B_{0} \in \mathcal{H}$, assumed smooth. 
 
 Here the weak formulation reads: for any $\Psi \in C^{1} ( [0,T]; C^{1}_{c} (\R^{3})$, for any $T>0$, 
 \begin{eqnarray}
\label{wedemerde} 
- \int_{0}^{T }  \int_{\R^{3}} (\partial_{t}  \Psi ) \cdot B_{2}   \, dx \,  dt
+ \big(  \int_{\R^{3}}    \Psi  \cdot B_{2}  \big) \vert_{t=T}  \, dx -  \big(  \int_{\R^{3}}    \Psi \vert_{t=0} \cdot B_{0} \, dx  \big)
\\ \nonumber + \int_{0}^{T }  \int_{\R^{3}} (\curl  \Psi ) \cdot  \big(  (\curl     B_{2}  )  \wedge B_{2}     \big)  \, dx \,  dt
 = -  \int_{0}^{T }  \int_{\R^{3}} (\curl  \Psi ) \cdot   (\curl     B_{2}  )   \, dx \,  dt,
\end{eqnarray}
and the energy inequality: for almost every $T>0$, 
\begin{equation}
\label{wener}
J_{\rm HMHD}  [ B_{2}](T) := 
\frac12  \big(  \int_{\R^{3}}  B_{2}  \, dx \big) \vert_{t=T} 
+  \int_{0}^{T }  \int_{\R^{3}}   (\curl     B_{2}  )^{2}  \, dx \,  dt
\leqslant 
\frac12   \int_{\R^{3}}  B_0^{2} \, dx. 
\end{equation}
We also consider a regular solution $B_{2}$ of \eqref{HallFort1}-\eqref{HallFort1} on $(0,T_{0})$, for $T_{0} > 0$.
We denote $B := B_{1} -B_{2}$ and expand $J_{\rm HMHD}  [ B](T)$ into 
\begin{equation*}
J_{\rm HMHD}  [ B](T) =
J_{\rm HMHD}  [ B_1](T) + J_{\rm HMHD}  [ B_{2}](T)
-   (\int_{\R^{3}}   B_{1} \cdot  B_{2} \, dx )(T) 
-2  \int_{0}^{T }  \int_{\R^{3}}  (\curl     B_{1}  )\cdot  (\curl     B_{2}  )  \, dx \,  dt,
\end{equation*}
and then we use that both $B_{1}$ and $B_{2}$ satisfy  the weak energy inequality \eqref{wener} to deduce that
\begin{equation}
\label{ducoin}
J_{\rm HMHD}  [ B](T)
\leqslant   \int_{\R^{3}}  B_0^{2} \, dx
-  (\int_{\R^{3}}   B_{1} \cdot  B_{2}\, dx  )(T) 
-2  \int_{0}^{T }  \int_{\R^{3}}  (\curl     B_{1}  )\cdot  (\curl     B_{2}  )   \, dx \,  dt .
\end{equation}
We use \eqref{wedemerde} with  $\Psi = B_{1}$ to get 
 \begin{eqnarray}
\label{wedemerde2} 
-    \int_{\R^{3}}  B_{0}^{2} 
+ \big(  \int_{\R^{3}}      B_{1} \cdot B_{2} \, dx \big) (T)  
+ \int_{0}^{T }  \int_{\R^{3}} (\curl  B_{1} ) \cdot   (\curl     B_{2}  )   \, dx \,  dt
 = 
\int_{0}^{T }  \int_{\R^{3}} (\partial_{t}  B_{1} ) \cdot B_{2}   \, dx \,  dt
\\ \nonumber - \int_{0}^{T }  \int_{\R^{3}} (\curl  B_{1} ) \cdot  \big(  (\curl     B_{2}  )  \wedge B_{2}     \big)  \, dx \,  dt
 .
\end{eqnarray}
Now we use that $B_{1}$ satisfies \eqref{HallFort1}-\eqref{HallFort2} to obtain
\begin{eqnarray*}
\int_{0}^{T }  \int_{\R^{3}} (\partial_{t}  B_{1} ) \cdot B_{2}   \, dx \,  dt
= -  \int_{0}^{T }  \int_{\R^{3}} \big( (\curl  B_{1})  \wedge B_{1} )\big) \cdot   \curl     B_{2}      \, dx \,  dt
- \int_{0}^{T }  \int_{\R^{3}} (\curl  B_{1})  \cdot (\curl  B_{2})  \, dx \,  dt .
\end{eqnarray*}
Plugging this into \eqref{wedemerde2} provides that 
\begin{eqnarray*}
L := -    \int_{\R^{3}}  B_{0}^{2} 
+ \big(  \int_{\R^{3}}      B_{1} \cdot B_{2}  \big) (T)  \, dx
+ 2\int_{0}^{T }  \int_{\R^{3}} (\curl  B_{1} ) \cdot   (\curl     B_{2}  )   \, dx \,  dt
\end{eqnarray*}
is given by
\begin{eqnarray*}
L  &=&
-  \int_{0}^{T }  \int_{\R^{3}} \big( (\curl  B_{1})  \wedge B_{1} )\big) \cdot   \curl     B_{2}    \, dx \,  dt
-  \int_{0}^{T }  \int_{\R^{3}} (\curl  B_{1})  \cdot (\curl  B_{2}  \wedge B_{2})  \, dx \,  dt
\\ &=&
-  \int_{0}^{T }  \int_{\R^{3}} \det( (\curl  B_{1})  , B_{1} , \curl     B_{2}  ) \, dx \,  dt
 +  \int_{0}^{T }  \int_{\R^{3}}  \det( (\curl  B_{1})  , B_{2} , \curl     B_{2}  ) \, dx \,  dt
 \\ &=&
-  \int_{0}^{T }  \int_{\R^{3}} \det( (\curl  B_{1})  , B , \curl     B_{2}  )  \, dx \,  dt
\\ &=&
-  \int_{0}^{T }  \int_{\R^{3}} \det( (\curl  B_{1})  , B , \curl     B  )  \, dx \,  dt,
\end{eqnarray*}
and therefore combining with \eqref{ducoin} we get 
\begin{eqnarray*}
J_{\rm HMHD}  [ B](T)
&\leqslant& 
  \int_{0}^{T }  \int_{\R^{3}} \det( (\curl  B_{1})  , B , \curl     B  ) \, dx \,  dt
\\ &\leqslant&  
C \int_{0}^{T }  \int_{\R^{3}} B^{2}  \, dx \,  dt+
\frac12  \int_{0}^{T }  \int_{\R^{3}}  (\curl B)^{2}  \, dx \,  dt , 
\end{eqnarray*}
which leads to the conclusion, thanks to a Gronwall estimate.

%%%%%%%%%%%%%
%%%%%%%%%%%%%
%%%%%%%%%%%%%
%%%%%%%%%%%%%

\section{Local conservations: Proof of Part i) of Theorem \ref{LLMonsag},  
Theorem~\ref{Hallonsaghel} and Theorem \ref{Hallonsagen}}
\label{Localconservations}
We use repetitively in the sequel the following formula: 
for two smooth vector fields $v$ and $w$ there holds
\begin{equation}
\label{formu}
-v \cdot \curl w + w \cdot  \curl v = \Div (v \wedge w) .
\end{equation}
\subsection{MLL equations: Proof of Part i) of Theorem \ref{LLMonsag}}
We first take the convolution of the equations with  the mollifier $\psi^\eps$ in order to obtain the regularized equations:
\begin{eqnarray}
\label{LLM1eps} 
\dt m_{\eps} +  (m \wedge \dt m)_{\eps} &=&
2 \sum_i \partial_i  \Big( m \wedge \partial_i m  \Big)_{\eps} + 2 (m \wedge H)_{\eps} ,
\\ \label{LLMGC2eps} 
\dt H_{\eps} + \curl  E_{\eps} &=& - \dt m_{\eps} ,
\\ \label{LLMGC3eps}
\dt E_{\eps} - \curl H_{\eps} &=& 0 , 
\\ \label{LLMGC4eps}
\Div  E_{\eps} &=& \Div (H_{\eps} + m_{\eps}) = 0 ,
\end{eqnarray}
where we use notation \eqref{notaeps}.

Next we apply formula  \eqref{convol1}, so that 
\eqref{LLM1eps} becomes 
\begin{equation}
\label{LLMGC1eps}
\dt m_{\eps} +  m_{\eps} \wedge \dt m_{\eps}
+ \mathcal{B}^{\eps} [m ,  \dt m] 
=  2 m_{\eps} \wedge \Delta m_{\eps} 
+ 2 \mathcal{B}^{\eps} [m ,   \Delta m] 
 + 2 m_{\eps} \wedge H_{\eps} 
 + 2 \mathcal{B}^{\eps} [m ,  H ] .
 \end{equation}
Let us take the inner product of \eqref{LLMGC1eps} with $\dt m_{\eps}$, 
$\Delta m_{\eps} $ and $H_{\eps} $ to get
\begin{eqnarray}
\label{R1}
|\dt m_{\eps} |^{2}
+ \mathcal{B}^{\eps} [m ,  \dt m] \cdot \dt m_{\eps}
&=&  2 (m_{\eps} \wedge \Delta m_{\eps} )  \cdot \dt m_{\eps}
+ 2 \mathcal{B}^{\eps} [m ,   \Delta m] \cdot \dt m_{\eps}
 \\ \nonumber && + 2 (m_{\eps} \wedge H_{\eps} )  \cdot \dt m_{\eps}
 + 2 \mathcal{B}^{\eps} [m ,  H ] \cdot \dt m_{\eps}  ,
 \end{eqnarray}
\begin{eqnarray} \label{R2}
 \dt m_{\eps} \cdot \Delta m_{\eps}  +  (m_{\eps} \wedge \dt m_{\eps} )\cdot \Delta m_{\eps} 
+ \mathcal{B}^{\eps} [m ,  \dt m]  \cdot \Delta m_{\eps}  
&=& 2 \mathcal{B}^{\eps} [m ,   \Delta m]  \cdot \Delta m_{\eps}
  \\ \nonumber && +  2 (m_{\eps} \wedge H_{\eps} ) \cdot \Delta m_{\eps}
 + 2 \mathcal{B}^{\eps} [m ,  H ] \cdot \Delta m_{\eps}  ,
 \end{eqnarray}
\begin{eqnarray}
\label{R3}
 \dt m_{\eps}  \cdot H_{\eps} +  (m_{\eps} \wedge \dt m_{\eps} ) \cdot H_{\eps} 
+ \mathcal{B}^{\eps} [m ,  \dt m]   \cdot H_{\eps} 
&=&  2 (m_{\eps} \wedge \Delta m_{\eps} ) \cdot H_{\eps} 
+ 2 \mathcal{B}^{\eps} [m ,   \Delta m]  \cdot H_{\eps} 
  \\ \nonumber && + 2 (m_{\eps} \wedge H_{\eps} )  \cdot H_{\eps} 
 + 2 \mathcal{B}^{\eps} [m ,  H ]   \cdot H_{\eps} . 
\end{eqnarray}
On the other hand we take the inner product of \eqref{LLMGC2eps} with $H_{\eps}$, the inner product of \eqref{LLMGC3eps} with $E_{\eps}$ and we take the sum to get
\begin{equation}
\label{R4}
\dt  ( |E_{\eps}|^{2} + |H_{\eps}|^{2} ) + 2 \Div (H_{\eps} \wedge E_{\eps}) = 
- 2 H_{\eps} \cdot \dt m_{\eps} .
\end{equation}

Let us now compute $\eqref{R1} - 2 \eqref{R2} - 2 \eqref{R3} +  \eqref{R4}$. 
This yields
\begin{equation*}
|\dt m_{\eps} |^{2} 
- 2 \dt m_{\eps} \cdot \Delta m_{\eps} 
+\dt  ( |E_{\eps}|^{2} + |H_{\eps}|^{2} ) + 2 \Div (H_{\eps} \wedge E_{\eps}) 
=  d_{\rm MLL}^{\mathfrak{a},\eps}  ,
\end{equation*}
with $d_{\rm MLL}^{\mathfrak{a},\eps} $ given by  \eqref{eMLLM}.

Therefore it suffices to observe that 
\begin{equation*}
- 2 \dt m \cdot \Delta m = - 2 \sum_i \partial_i 
\Big(   \dt m \cdot \partial_i m \Big) + \dt  ( |\nabla m|^{2} ) ,
\end{equation*}
to obtain
\begin{equation}
\label{LLlocaleps}
\dt e^{\eps}_{\rm MLL} + 
d^{\eps}_{\rm MLL}  + 
\Div f^{\eps}_{\rm MLL} 
= d^{\mathfrak{a},\eps}_{\rm MLL} ,
\end{equation}
where
\begin{equation}
\label{edfMLLeps}
e^{\eps}_{\rm MLL}  :=  |E_\eps |^{2} + |H_\eps |^{2} + |\nabla m_\eps |^{2} ,
\quad d^{\eps}_{\rm MLL}  := 
|\dt m_\eps |^{2} 
\quad  f^{\eps}_{\rm MLL}  := 
- 2 (   \dt m_\eps \cdot \partial_i m_\eps )_{i=1,2,3}
 + 2 H_\eps \wedge E_\eps .
\end{equation}
Now we prove that, when $\eps \rightarrow 0$, $d_{\rm MLL}^{\mathfrak{a},\eps}$ 
converges, in the sense of distributions, to  $d^\mathfrak{a}_{\rm MLL} $
whatever is the choice of the mollifier. Let us recall that $d^\mathfrak{a}_{\rm MLL} $ is given by \eqref{LLlocal}.

Indeed it follows from the regularity of $m$, $E$ and $H$ that 
$|\dt m_\eps |^{2} $, $|E_\eps|^{2}$, $ |H_\eps|^{2}$, $ |\nabla m_\eps|^{2}$, 
$ \dt m_\eps \cdot \partial_i m_\eps $ and $H_\eps \wedge E_\eps$ 
converge respectively to $|\dt m |^{2} $, $|E|^{2}$, $ |H|^{2}$, 
$ |\nabla m|^{2}$, $ \dt m \cdot \partial_i m $ and $H \wedge E$ in 
$L^1 ((0,T) \times \R^3 )_{\rm loc}$.
As a consequence the left hand side of \eqref{LLlocaleps} converges, 
in the sense of distributions, to
$$|\dt m |^{2} 
+ \dt  ( |E|^{2} + |H|^{2} + |\nabla m|^{2} )
- 2 \sum_i \partial_i \Big(   \dt m \cdot \partial_i m \Big)
 + 2 \Div (H \wedge E)  . $$
This entails that,  $d_{\rm MLL}^{\mathfrak{a},\eps} $ converges, in the sense of distributions, to $d^\mathfrak{a}_{\rm MLL} $.

%%%%%%%%%%
%%%%%%%%%%
%%%%%%%%%%
%%%%%%%%%%
%%%%%%%%%%

\subsection{HMHD equations: Proof of Part i) of Theorem \ref{Hallonsaghel} 
and Theorem~\ref{Hallonsagen}}
\subsubsection{Regularization}
We start with the regularized equations:
\begin{eqnarray}
\label{HMHD1Q}
\dt u_{\eps} +   \Div (u \otimes u - B\otimes B )_{\eps} + \nabla  (p_m )_{\eps}  &=& \Delta u_{\eps} ,
\\ \label{HMHD2Q}
\Div u_{\eps} &=& 0 ,
\\ \label{HMHD3Q}
\dt B_{\eps} - \curl  (u  \wedge  B )_{\eps} + \curl  \Div (B\otimes B)_{\eps} &=& \Delta B_{\eps} ,
\\ \label{HMHD4Q}
\Div  B_{\eps} &=& 0 .
\end{eqnarray}
We use the decompositions \eqref{convol0}, \eqref{convol1} and \eqref{convol2}  to recast \eqref{HMHD1Q} and \eqref{HMHD3Q} as follows:
\begin{eqnarray}
\label{HMHD1Qsplit}
\dt u_{\eps} +   \Div (u_{\eps} \otimes u_{\eps} - B_{\eps} \otimes     B_{\eps}) +   \Div  ( \mathcal{C}^{\eps} [u,u ]  -    \mathcal{C}^{\eps} [B,B  ]  )
+ \nabla ( p_{\eps}   + \frac12 | B_{\eps} |^2 ) +  \frac12  \nabla \mathcal{A}^{\eps} [B,B  ]   =  \Delta u_{\eps} ,
\\ \label{HMHD3Qsplit}
\dt B_{\eps} - \curl  (u_{\eps}  \wedge  B_{\eps} ) 
- \curl   \mathcal{B}^{\eps} [u,B ]
+ \curl  \Big(  ( \curl B_{\eps} )  \wedge B_{\eps} +  \Div \mathcal{C}^{\eps} [B,B  ]   \Big)   = \Delta B_{\eps}  .
\end{eqnarray}
\subsubsection{Local magneto-helicity identity: Proof of Part i) 
of Theorem \ref{Hallonsaghel}} 
\label{Localconservhel}
Thanks to Leibniz' identity and \eqref{formu} there holds
\begin{equation*}
\dt (A_{\eps} \cdot B_{\eps} ) = ( \dt A_{\eps} )\cdot B_{\eps}  + A_{\eps} \cdot \dt  B_{\eps} = 2 A_{\eps} \cdot  \dt   B_{\eps} +  \Div (A_{\eps}  \wedge \dt A_{\eps} ) .
\end{equation*}
Using now  \eqref{HMHD3Qsplit} we obtain
\begin{eqnarray*}
\dt (A_{\eps} \cdot B_{\eps} ) &=&
- 2 A_{\eps}  \cdot  \curl \Big( (\curl B_{\eps} -u_{\eps})  \wedge  B_{\eps} )  \Big)
+ 2 A_{\eps}  \cdot  \Delta B_{\eps}
+ \Div (A_{\eps}  \wedge \dt A_{\eps} )
 \\ && + 2 A_{\eps}  \cdot  \curl   \mathcal{B}^{\eps} [u,B ]
  -  2 A_{\eps}  \cdot   \curl  \Div \mathcal{C}^{\eps} [B,B  ] .
\end{eqnarray*}
Thanks to \eqref{formu} and to the divergence free conditions we obtain
\begin{equation*}
- A_{\eps}  \cdot  \curl  \Big((\curl B_{\eps} -u_{\eps} ) \wedge B_{\eps} ) \Big) =  \Div \Big( A_{\eps} \wedge  ( (\curl B_{\eps} -u_{\eps} )    \wedge B_{\eps} ) \Big) ,
\end{equation*}
and
\begin{equation*}
  A_{\eps}  \cdot  \Delta B_{\eps} = - A_{\eps}  \cdot \curl \curl  B_{\eps}  = - B_{\eps} \cdot \curl B_{\eps}
+  \Div (A_{\eps} \wedge  \curl B_{\eps} ).
\end{equation*}

Therefore
\begin{equation}
\label{HelEps}
\dt  (A_{\eps} \cdot B_{\eps}) +  2 B_{\eps} \cdot  \curl B_{\eps}  
- \Div \Big( 2 \big( (u_{\eps}-\curl B_{\eps}) \wedge B_{\eps} 
- 2 \curl B_{\eps} - \partial_t A_{\eps} \big) \wedge A_{\eps} \Big)
  =  d_m^{\mathfrak{a},\eps}  ,
\end{equation}
where $d_m^{\mathfrak{a},\eps}$ is given by \eqref{eli}.

Let us now prove that  $d_m^{\mathfrak{a},\eps} $  converges 
in the sense of distributions to $d_m^{\mathfrak{a}}$.
Indeed we are going to prove that the left hand side of \eqref{HelEps} 
converges to
$$ \dt  (A \cdot B) +  2 B \cdot  \curl B  - \Div \Big( 2(u-\curl B) \wedge B 
- 2 \curl B - \partial_t A ) \wedge A \Big) .$$
Actually thanks to the estimates given by the existence theorem \ref{ADFL}, 
and the fact that the vector potential $A$ is in $L^2 (0,T;H^2 (\R^3 ))$ 
by elliptic regularity, we easily infer that
$$ \dt  (A_{\eps} \cdot B_{\eps}) +  2 B_{\eps} \cdot  \curl B_{\eps}  - \Div \Big( \big( 2(u_{\eps}-\curl B_{\eps}) \wedge B_{\eps} - 2 \curl B_{\eps}  \big) \wedge A_{\eps} \Big)$$
converges in the sense of distributions to
$$ \dt  (A \cdot B) +  2 B \cdot  \curl B  - \Div \Big(  \big(2(u-\curl B) \wedge B - 2 \curl B  \big) \wedge A \Big) .$$
Let us now turn our attention to the last term of  the left hand side of \eqref{HelEps}.
Using again  elliptic regularity,
we easily infer that  $A$ is in $L^\infty (0,T; \mathcal{V})$.
Moreover using equation \eqref{0HMHD3}, one infers that $\dt  A$ 
is in $L^\frac{4}{3} (0,T; \mathcal{V}')$.
 Actually this  estimate is used in course of proving  the existence theorem \ref{ADFL}, see \cite{HMHD}.
 From that we deduce that $ \Div \Big(  (\partial_t A_{\eps} ) \wedge A_{\eps} \Big)$ converges  in the sense of distributions to $ \Div \Big(  (\partial_t A ) \wedge A \Big)$.

This concludes the proof of the first part of Theorem \ref{Hallonsaghel}.

%%%%%%%%%%
%%%%%%%%%%
%%%%%%%%%%
%%%%%%%%%%
%%%%%%%%%%

%
\subsubsection{Local Energy identity: Proof of Part i) 
of Theorem \ref{Hallonsagen}}
\label{LoEnHMHD}
Let us take the inner product of \eqref{HMHD1Qsplit} and \eqref{HMHD3Qsplit} respectively with $u_{\eps}$ and $B_{\eps}$, and sum the resulting identities, taking into account that, thanks to \eqref{HMHD2Q} and \eqref{HMHD4Q}, 
\begin{equation*}
u_{\eps} \cdot \Div (u_{\eps} \otimes u_{\eps}) + u_{\eps} \cdot \nabla p_{\eps} = 
\Div \Big(  ( \frac12 |  u_{\eps} |^2 + p_{\eps} ) u_{\eps} \Big),
\end{equation*}
\begin{equation*}
u_{\eps}   \cdot  \Big(    - \Div (B_{\eps} \otimes     B_{\eps})  +   \frac{1}{2} \nabla (   | B_{\eps} |^2 )   \Big) 
 = 
-     u_{\eps}   \cdot   \Big(  ( \curl B_{\eps} )  \wedge B_{\eps} \Big) 
=  ( \curl B_{\eps} )  \cdot   \Big( u_{\eps}  \wedge  B_{\eps}  \Big),
\end{equation*}
so that
\begin{equation*}
u_{\eps}   \cdot  \Big(    - \Div (B_{\eps} \otimes     B_{\eps})  +   \frac{1}{2} \nabla (   | B_{\eps} |^2 )   \Big) 
  - B_{\eps}  \cdot  \Big( \curl  (u_{\eps}  \wedge  B_{\eps} )  \Big)
  = \Div  \Big( B_{\eps}    \wedge       (u_{\eps}  \wedge  B_{\eps} )          \Big)     ,
\end{equation*}
\begin{equation*}
 B_{\eps}  \cdot  \Big( \curl  \Big(  ( \curl B_{\eps} )  \wedge B_{\eps} \Big)    \Big)    =
   \Div  \Big(   \Big(  ( \curl B_{\eps} )  \wedge B_{\eps} \Big) \wedge B_{\eps} \Big) ,
\end{equation*}
and 
\begin{equation*}
-  u_{\eps}   \cdot  \Delta u_{\eps}  - B_{\eps}   \cdot  \Delta B_{\eps} =  |\curl u_{\eps}  |^2 + |\curl B_{\eps}  |^2
+ \Div \Big( (\curl u_{\eps} )\wedge   u_{\eps}  +  (\curl B_{\eps} )\wedge   B_{\eps}   \Big)  .
\end{equation*}
We thus obtain
\begin{equation}
\label{cucu}
\dt e^{\eps}_{\rm HMHD} + 
d^{\eps}_{\rm HMHD}  + 
\Div f^{\eps}_{\rm HMHD} 
=  d^{\mathfrak{a},\eps}_{\rm HMHD} ,
\end{equation}
where
\begin{gather}
\label{defcucu}
e^{\eps}_{\rm HMHD} := \frac{1}{2}  \Big( |u_{\eps} |^{2} +   |B_{\eps} |^{2}   \Big) , \quad 
 d^{\eps}_{\rm HMHD} := |\curl u_{\eps}  |^2 + |\curl B_{\eps}  |^2 , \quad 
\\ \nonumber f^{\eps}_{\rm HMHD} :=  ( \frac12 |  u_{\eps} |^2 + p_{\eps} ) u_{\eps}
 + B_{\eps} \wedge  ( u_{\eps} \wedge B_{\eps} )
 + ( \curl u_{\eps} )  \wedge u_{\eps}
 +  ( \curl B_{\eps} )  \wedge B_{\eps}
+  \big(  ( \curl B_{\eps} )  \wedge B_{\eps} \big) \wedge B_{\eps} ,
\\  \nonumber d^{\mathfrak{a},\eps}_{\rm HMHD}  :=
 - u_{\eps} \cdot   \Div  \Big( \mathcal{C}^{\eps} [u,u ] - \mathcal{C}^{\eps} [B,B  ]     \Big) 
 -   \frac{1}{2} u_{\eps} \cdot  \nabla \mathcal{A}^{\eps} [B, B ]
+ B_{\eps} \cdot  \curl   \mathcal{B}^{\eps} [u,B ] 
- B_{\eps} \cdot   \curl   \Div \mathcal{C}^{\eps} [B,B  ]    .
\end{gather}

Let us now prove that if $u$ and $B$ are given by Theorem~\ref{ADFL}
with $B \in L^{4} ((0,T) \times \R^{3} )_{\rm loc}$ 
then $ d_{\rm HMHD}^{\mathfrak{a},\eps} $ converges, in the sense of distributions, 
to $d_{\rm HMHD}^{\mathfrak{a}} $. 

First observe that since $u$ and $B$ belong to 
$L^\infty(0,T;L^2(\R^3)) \cap L^2(0,T;H^1(\R^3))$, 
they also belong to $L^3((0,T)\times\R^3)$.

It is not difficult to see that, under these assumptions, 
\begin{eqnarray*}
  |u_{\eps} |^{2} +   |B_{\eps} |^{2} ,
   |  u_{\eps} |^2  u_{\eps}, B_{\eps} \wedge  ( u_{\eps} \wedge B_{\eps} ), 
 ( \curl u_{\eps} )  \wedge u_{\eps},
  ( \curl B_{\eps} )  \wedge B_{\eps},
\big(  ( \curl B_{\eps} )  \wedge B_{\eps} \big) \wedge B_{\eps},
\\  |\curl u_{\eps}  |^2  \text{ and }  |\curl B_{\eps}  |^2
\end{eqnarray*}
 converge in $L^1 ((0,T) \times \R^3 )_{\rm loc}$ respectively to 
\begin{eqnarray*}
  |u |^{2} +   |B |^{2} ,
   |  u |^2  u, B \wedge  ( u \wedge B ), 
 ( \curl u )  \wedge u,
  ( \curl B )  \wedge B,
\big(  ( \curl B )  \wedge B \big) \wedge B,
   |\curl u  |^2  \text{ and }  |\curl B  |^2 .
\end{eqnarray*}
Therefore in order to prove that the left hand side of \eqref{cucu} converges 
in  the sense of distributions to the  left hand side of \eqref{HMHDlocalsansa}
it is sufficient to prove that $ p_{\eps} $ converges in  $L^\frac32 ((0,T) \times \R^3 )_{\rm loc}$
to $p$. 
But taking the divergence of \eqref{HMHD1}, and taking \eqref{HMHD2} 
into account, we obtain that the magnetic  pressure $p_m $ satisfies
\begin{equation}
\label{ellpm}
 \Delta p_m = -  \Div \Div (u \otimes u - B\otimes B ) .
\end{equation}
Then classical elliptic regularity allows to conclude that  $p_m $, 
and therefore $p$, is in $L^\frac32 ((0,T) \times \R^3 )_{\rm loc}$, 
what concludes the proof of Part i) of Theorem \ref{Hallonsagen}.

%%%%%%%%%%
%%%%%%%%%%
%%%%%%%%%%
%%%%%%%%%%
%%%%%%%%%%

\section{Technicalities}
\label{secTech}
In this section we gather a few technical results which will be 
useful in the sequel. 
\subsection{Some injections}

The following lemma is a consequence of Bernstein's lemma. 
Its proof is given in the Appendix. 
\begin{Lemma}
\label{bern}
Let  $\alpha,\tilde{\alpha} \in (0,1) \cup (1,2)$, and $p,r \in [1,\infty]$. 
Assume that $\tilde{\alpha}\leq\alpha$ and define $\tilde{p}\in[1,\infty]$ 
by $\tilde{\alpha} - 3/\tilde{p} = \alpha -3/p$. 
Then $\widetilde{L}^r (0,T; \dot{B}^\alpha_{p,\infty} (\R^3))$ 
is continuously embedded in 
$\widetilde{L}^r (0,T; \dot{B}^{\tilde{\alpha}}_{\tilde{p},\infty} (\R^3))$. 
As a consequence, any 
$u\in\widetilde{L}^r (0,T; \dot{B}^\alpha_{p,\infty} (\R^3))_{\rm loc}$
belongs to 
$\widetilde{L}^r (0,T; \dot{B}^{\tilde{\alpha}}_{\tilde{p},\infty} 
(\R^3))_{\rm loc}$. 
\end{Lemma}
Let us observe that if $(p,\alpha)$ satisfies the relationship 
in \eqref{ship} then so does any $(\tilde{p},\tilde{\alpha})$ 
such that $\tilde{\alpha} - 3/\tilde{p} = \alpha - 3/p$.
Similarly if $(q,\beta)$ satisfies $q=12/(4\beta-1)$ 
(see Remark~\ref{interpol}),  
then so does any $(\tilde{q},\tilde{\beta})$ such that 
$\tilde{\beta} - 3/\tilde{q} = \beta - 3/q$. 

We will use in particular that 
\begin{itemize}
\item for $\alpha\in[4/3,11/6]$ and $\displaystyle p := \frac{9}{3\alpha-1}$, 
a function belonging to  the space 
$\widetilde{L}^3(0,T;\dot{B}^{\alpha}_{p,\infty}(\R^3))_{\rm loc}$ 
also belongs to  the space 
$\widetilde{L}^3 (0,T;\dot{B}^{\tilde{\alpha}}_{\tilde{p},\infty} 
(\R^3))_{\rm loc}$, 
with $\tilde{\alpha} := 4 -2 \alpha$ and  
$\displaystyle \tilde{p} := \frac{9}{3\tilde{\alpha} - 1}$ 
which satisfies 
$\displaystyle \frac{2}{p} + \frac{1}{\tilde{p}}  = 1$.
\item for $\beta\in(9/8,3/2)$ and $\displaystyle q := \frac{12}{4\beta-1}$, 
a function belonging to  the space 
$\widetilde{L}^4 (0,T ; \dot{B}^\beta_{q,\infty} (\R^3))_{\rm loc}$ 
also belongs to   the space 
$\widetilde{L}^4 (0,T ; \dot{B}^{2-\beta}_{\tilde{q},\infty} (\R^3))_{\rm loc}$, 
with $\displaystyle \tilde{q} := \frac{3}{\frac{7}{4}-\beta}$ 
which satisfies 
$\displaystyle \frac{1}{q} + \frac{1}{\tilde{q}}  = \frac{1}{2}$.
\end{itemize}

%%%%%%%%%%
%%%%%%%%%%
%%%%%%%%%%
%%%%%%%%%%
%%%%%%%%%%

\subsection{A Constantin-E-Titi type lemma}
\label{CETtype}
We will make a crucial use of the following lemma adapted from \cite{CET}. 
The notations $\mathcal{A}^{\eps}, \mathcal{B}^{\eps}, \mathcal{C}^{\eps}$ 
are from \eqref{convol0}, \eqref{convol1}, \eqref{convol2}.

\begin{Lemma}
\label{cruciallemma}
Let $i=0$, $1$ or $2$. 
Let $(r_1,r_2,r_3) \in [1,\infty]^3$, $(p_1,p_2,p_3) \in [1,\infty]^3$ 
and $(\alpha_1,\alpha_2,\alpha_3) \in [0,1)^3$ be such that 
$$
\frac{1}{r_{1}} + \frac{1}{r_{2}}  + \frac{1}{r_{3}} = 
\frac{1}{p_{1}} + \frac{1}{p_{2}}  + \frac{1}{p_{3}} = 1 
\quad \text{and} \quad 
\alpha_{1} + \alpha_{2} +\alpha_{3} \geq i. 
$$
Let $\phi^1 , \phi^2 , \phi^3$ be functions on $(0,T) \times \R^3$ 
such that for $j = 1,2,3$, 
\begin{itemize}
\item either  $\alpha_{j} \in (0,1)$ and $\phi^j \in 
\widetilde{L}^{r_j} (0,T ; \dot{B}^{\alpha_{j}}_{p_{j},\infty} (\R^3 ))_{\rm loc}$,
\item or $\alpha_{j} = 0$, and 
$\phi^j \in L^{r_j}  (0,T ; L^{p_{j}} (\R^3))_{\rm loc}$,
 \end{itemize}
and such that for at least one index $j \in \{1,2,3\}$ 
($j \in \{1,2\}$ in the case where $\alpha_{1} + \alpha_{2} +\alpha_{3} = i=0$),   
\begin{itemize}
\item either  $\alpha_{j} \in (0,1)$ and $\phi^j \in 
\widetilde{L}^{r_j} (0,T ; \dot{B}^{\alpha_{j}}_{p_{j},c_{0}} (\R^3))_{\rm loc}$,
\item or $\alpha_{j} = 0$, $r_{j}, p_{j} < \infty$ and 
$\phi^j \in L^{r_j}  (0,T ; L^{p_{j}}  (\R^3 ))_{\rm loc}$.
\end{itemize}
Then, for all $\chi \in C^\infty_{\rm c}((0,T)\times\R^3)$,
\begin{eqnarray}
\label{trili}
\int_{0}^{T} \int_{\R^3} \chi \Big(| \mathcal{A}^{\eps} [\phi^1,\phi^2] |   
+  | \mathcal{B}^{\eps} [\phi^1,\phi^2]   |   
+  | \mathcal{C}^{\eps} [\phi^1,\phi^2]   |  \Big)  \, 
| \nabla^{i}  \phi^3_{\eps}  | \,  dx \, dt \rightarrow 0 
\quad \text{when } \eps \rightarrow 0 .
\end{eqnarray}
\end{Lemma}

Before proving Lemma \ref{cruciallemma}, we start with a few 
preliminary results. 
\begin{Lemma} \label{approx}
Let $\alpha \in (0,1)$, and $p,r \in [1,\infty]$. 
\begin{enumerate}
\item For all $u\in \widetilde{L}^r (0,T ; \dot{B}^\alpha_{p,\infty} (\R^3))$, 
\begin{equation}
\label{DeuzO}
\| u - u_\eps \|_{L^{r} ( (0,T) ; L^{p} ( \R^3 ))} =  O(\eps^\alpha ) .
\end{equation}
\item For all $u \in \widetilde{L}^r (0,T ; \dot{B}^\alpha_{p,c_0} (\R^3))$, 
\begin{equation}
\label{Deuz}
\| u - u_\eps \|_{L^{r} ( (0,T) ; L^{p} ( \R^3 ))} =  o(\eps^\alpha ) .
\end{equation}
\item If $r,p \in [1,\infty)$,  then 
for all $u \in L^{r} (0,T ; L^{p} (\R^3))$, 
\begin{equation}
\label{DeuzLp}
\|   u - u_\eps  \|_{L^{r} ( (0,T) ; L^{p} ( \R^3 ))} =  o(1) .
\end{equation}
\end{enumerate}
\end{Lemma}
\begin{proof}
We begin with the proof of \eqref{DeuzO}.
We use that 
\begin{equation*}
u_\eps (t,x) - u (t,x) =  
\int_{\R^3} \psi^{\eps} (y)  ( u(t,x-y) - u(t,x) ) dy ,
\end{equation*}
so that, with $\delta_y u(t,x) = u(t,x-y) - u(t,x)$, 
\begin{eqnarray} \label{u-uepsLrLp}
\| u - u_\eps \|_{L^r(0,T;L^{p} (\R^3))} \leq  
 \int_{\R^3} \psi^{\eps} (y)  \| \delta_y u \|_{L^r(0,T;L^{p} (\R^3))} dy .
\end{eqnarray}
Now, this implies 
$$
\| u - u_\eps \|_{L^r(0,T;L^{p} (\R^3))} \leq  
\left( \int_{\R^3} \psi^{\eps} (y) |y|^\alpha dy \right) 
\| u \|_{\widetilde{L}^r(0,T;B^\alpha_{p,\infty} (\R^3))} .
$$
According to the size of the support of $\psi^\eps$, 
this proves \eqref{DeuzO}. To obtain \eqref{Deuz}, we write 
\begin{eqnarray*}
\eps^{-\alpha} \| u - u_\eps \|_{L^r(0,T;L^{p} (\R^3))} 
\leq \int_{\R^3} \psi^{\eps} (y) 
\left( \frac{|y|}{\eps} \right)^\alpha \| f_{\alpha,p}[u](y) \|_{L^r(0,T)} dy 
\leq C \int_{\R^3} \psi^{\eps} (y) \| f_{\alpha,p}[u](y) \|_{L^r(0,T)} dy , 
\end{eqnarray*}
and again, the fact that supp$(\psi^\eps)$ has size $\eps$ yields the result. 

To prove \eqref{DeuzLp}, we come back to \eqref{u-uepsLrLp}. 
For all $t \in (0,T)$, 
$\| \delta_y u(t,\cdot) \|_{L^{p} (\R^3)} \leq 
2 \| u(t,\cdot) \|_{L^{p} (\R^3)}$, and since $p<\infty$, 
$\| \delta_y u (t,\cdot) \|_{L^{p} (\R^3)} \tendlorsque{y}{0} 0$, 
so that Lebesgue's dominated convergence Theorem concludes. 
\end{proof}
\begin{Lemma}
\label{lowf}
Let $\alpha \in (0,1)$, and $p,r \in [1,\infty]$.  
\begin{enumerate}
\item For all $u \in \widetilde{L}^r (0,T ; \dot{B}^\alpha_{p,\infty} (\R^3))$,
\begin{equation}
\label{OPI}
\| \nabla u_\eps \|_{ L^{r} (0,T ;L^{p} (\R^3))} = O(\eps^{\alpha - 1} )
\quad \text{and} \quad  
\| \nabla^{2} u_\eps  \|_{L^r (0,T ;L^{p} (\R^3))} = O(\eps^{\alpha - 2} ) .
\end{equation}
\item For all $u \in \widetilde{L}^{r} (0,T ; \dot{B}^\alpha_{p,c_{0}} (\R^3))$, 
\begin{equation}
\| \nabla u_\eps \|_{ L^r (0,T ;L^{p} (\R^3))} = o(\eps^{\alpha - 1} )
\quad \text{and} \quad  
\| \nabla^{2} u_\eps \|_{L^r (0,T ;L^{p}  (\R^3))} = o(\eps^{\alpha - 2} ) .
\end{equation}
\item If $r,p \in [1,\infty)$, then 
for all $u \in L^{r} ((0,T) ; L^{p} (\R^3))$, 
\begin{equation}
\| \nabla u_\eps \|_{ L^r (0,T ;L^{p} (\R^3))} = o(1/\eps)
\quad \text{and} \quad  
\| \nabla^{2} u_\eps \|_{L^r (0,T ;L^{p}  (\R^3))} = o(1/\eps^2) .
\end{equation}
\end{enumerate}
\end{Lemma}
\begin{proof}
Let us only prove the part of \eqref{OPI} regarding $\nabla u_\eps$, 
the other cases being similar. We start from 
\begin{equation*}
\nabla  u_\eps (t,x) =  \frac{1}{\eps} \int_{\R^3} \tilde{\psi}^{\eps} (y)  u(t,x-y) dy ,
\end{equation*}
where
\begin{equation*}
\tilde{\psi}^{\eps} (y) :=  \eps^{-3} \tilde{\psi} (\eps^{-1} x) ,  
\text{ whith } \tilde{\psi} := \nabla \psi .
\end{equation*}
Since the mean value of $\tilde{\psi}$ over $\R^3$ vanishes, we get
\begin{eqnarray*}
\nabla  u_\eps  (t,x) &=&   \frac{1}{\eps} 
\int_{\R^3} \tilde{\psi}^{\eps} (y)  (u(t,x-y) -u(t,x)) dy ,
\\                           
&=&   \frac{1}{\eps} \int_{\R^3}  | y |^{\alpha} \tilde{\psi}^{\eps} (y)   
\frac{u(t,x-y) -u(t,x)}{| y |^{\alpha}} dy ,
\end{eqnarray*}
and we conclude as in the proof of \eqref{DeuzO} above. 
\end{proof}

Let us now prove Lemma \ref{cruciallemma}.
\begin{proof}[Proof of Lemma \ref{cruciallemma}]
In order to prove the part of the claim concerning the vector product 
it  suffices to observe that $\mathcal{B}^{\eps} [\phi^1,\phi^2]$ may 
be written 
\begin{equation*}
\mathcal{B}^{\eps} [\phi^1,\phi^2] = r^{\eps} [\phi^1,\phi^2]  
- (\phi^1 - \phi^1_{\eps} ) \wedge (\phi^2 - \phi^2_{\eps} ) , 
\end{equation*}
where
\begin{equation*}
r^{\eps} [\phi^1,\phi^2] (x):= \int_{\R^3} \psi^{\eps} (y)  
\delta_y \phi^1 (x) \wedge \delta_y \phi^2 (x) \, dy .
\end{equation*}
Now, if $\chi \in C^\infty_{\rm c}((0,T)\times\R^3)$, there exists 
some nonnegative $\theta \in C^\infty_{\rm c}((0,T)\times\R^3)$ 
taking the value $1$ on supp$(\chi)$, so that $\chi=\chi\theta^3$. 
Then, we have 
\begin{gather*}
\int_{0}^{T} \int_{\R^3} \chi |r^{\eps} [\phi^1,\phi^2] |  \, 
| \nabla^{i}  \phi^3_{\eps}  | \,  dx \, dt = 
\\ \int_{(0,T)\times\R^3} \chi(t,x) 
\left| \int_{\R^3} \psi^{\eps} (y) ( (\theta \delta_y \phi^1) (t,x) ) 
\wedge ( (\theta \delta_y \phi^2) (t,x) ) \, dy \right| 
| (\theta \nabla^{i}  \phi^3_{\eps})(t,x)  | \,  dt \, dx .
\end{gather*}
But $\theta \delta_y \phi^1 = \delta_y (\theta\phi^1) - \phi^1 \delta_y \theta$.  
Since supp$(\psi^\eps)$ is contained in some ball of size $\eps$, and $\theta$ 
takes the value $1$ on supp$(\chi)$, for $\eps$ small enough, with $y$ in 
supp$(\psi^\eps)$, $\delta_y \theta$ vanishes on supp$(\chi)$: 
$$
\text{for } \eps \text{ small enough,} \quad 
\int_{0}^{T} \int_{\R^3} \chi |r^{\eps} [\phi^1,\phi^2] |  \, 
| \nabla^{i}  \phi^3_{\eps}  | \,  dx \, dt = 
\int_{0}^{T} \int_{\R^3} \chi |r^{\eps} [\theta\phi^1,\theta\phi^2] |  \, 
| \nabla^{i} (\theta \phi^3_{\eps}) | \,  dx \, dt.
$$

Then, use H\"older's inequality, estimating $r^{\eps} [\theta\phi^1,\theta\phi^2]$  
as in the proof of Lemma~\ref{approx}, and combine with Lemma~\ref{approx} 
and Lemma~\ref{lowf}, noticing that among the estimates, 
at least one of the $O$'s is a $o$. 
\end{proof}
%

%%%%%%%%%%
%%%%%%%%%%
%%%%%%%%%%
%%%%%%%%%%
%%%%%%%%%%

\section{Vanishing of anomalous energy dissipation for the MLL equations:  Proof of Part ii) of Theorem \ref{LLMonsag}}
\label{Va1}

In order to conclude the proof of Theorem \ref{LLMonsag}, it is sufficient 
to  prove that for all $\chi \in C^\infty_{\rm c}((0,T)\times\R^3)$, 
\begin{equation}
\label{tri}
\int_{(0,T) \times \R^3} \chi  d^{\mathfrak{a},\eps}_{\rm MLL} \, dx \, dt  
\rightarrow 0 , \quad \text{as } \eps \rightarrow 0 .
\end{equation}
For the sequel, we fix such a function $\chi$. 
We expand the local  anomalous energy dissipation $d^{\mathfrak{a},\eps}_{\rm MLL}$, defined  in \eqref{eMLLM}, into
\begin{eqnarray*}
d^{\mathfrak{a},\eps}_{\rm MLL} &=& 
- \mathcal{B}^{\eps} [m ,  \dt m -2 H ] \cdot ( \dt m_\eps  -2 H_\eps )
+ 2 \mathcal{B}^{\eps} [m ,   \Delta m] \cdot ( \dt m_\eps -2 H_\eps )
\\ &&+ 2 \mathcal{B}^{\eps} [m ,  \dt m -2 H]  \cdot \Delta m_\eps 
- 4 \mathcal{B}^{\eps} [m ,   \Delta m]  \cdot \Delta m_\eps ,
\end{eqnarray*}
which we denote 
$$
d^{\mathfrak{a},\eps}_{\rm MLL}  = F^\eps_1 [m,H] + F^\eps_2 [m,H]+ F^\eps_3 [m,H]+ F^\eps_4 [m].
$$
Now it suffices to proceed as follows. 
\begin{Remark} \label{interpol}
When $m \in \widetilde{L}^{3}_t \dot{B}^\alpha_{p,\infty} \cap L^\infty_{t,x}$ 
with $\alpha \in (3/2,2)$ and $p = 9/(3\alpha-1)$, by interpolation 
$m$ belongs to $\widetilde{L}^{4}_t \dot{B}^\beta_{q,\infty}$ with 
$\beta := 3\alpha/4 \in (9/8,3/2)$ and $q = 12/(4\beta-1)$. 
This follows from the injection 
$L^\infty ( (0,T) \times \R^3 ) \hookrightarrow 
\widetilde{L}^3 (0,T; B^{0}_{\infty,\infty} (\R^3))$ 
and Theorem $6.4.5$ in  \cite{BergLof}.
\end{Remark} 
\subparagraph{First term.}
Since $ \dt m -2 H \in L^{2} ((0,T) \times  \R^3 )$ and 
$m \in L^{\infty} ((0,T) \times  \R^3 )$,
we deduce from Lemma \ref{cruciallemma} (with $i=0$, $\phi^1=m$, 
$\phi^2 = \phi^3 = \dt m - 2H$) that 
\begin{equation}
\label{tri1}
\int_{(0,T) \times \R^3} F^\eps_1 [m,H]  \chi \, dx \, dt  \rightarrow 0,  
\text{ as } \eps \rightarrow 0 .
\end{equation}
\subparagraph{Second term.}
For the second term we use first an integration by parts to get
\begin{eqnarray*}
\int_{(0,T) \times \R^3} \chi F^\eps_2 [m,H]  \, dx \, dt &=&
- 2   \sum_k  \int_{(0,T) \times \R^3} \chi
\mathcal{B}^{\eps} [m , \partial_k  m] \cdot 
\partial_k   ( \dt m_\eps -2 H_\eps ) 
\, dx \, dt 
\\ &&- 2   \sum_k  \int_{(0,T) \times \R^3} (\partial_k \chi )
\mathcal{B}^{\eps} [m , \partial_k  m] \cdot 
   ( \dt m_\eps -2 H_\eps ) 
\, dx \, dt
\\ && \quad  := I^\eps_1 + I^\eps_2 . 
\end{eqnarray*}
Then, one uses  Lemma \ref{cruciallemma} 
with $i=1$, 
\begin{eqnarray*}
\phi^3 = \dt m - 2 H \in L^2  ((0,T) \times  \R^3 ), \quad
\phi^2 = \nabla m \in \widetilde{L}^4 (0,T ;\dot{B}^{\beta-1}_{q,\infty}( \R^3) )_{\rm loc} ,
\\ \text{ and } \phi^1 = m \in \widetilde{L}^4 (0,T ; \dot{B}^\beta_{q,\infty} ( \R^3) )_{\rm loc} \hookrightarrow 
\widetilde{L}^4 (0,T ;\dot{B}^{\tilde\beta}_{\tilde{q},\infty} ( \R^3) )_{\rm loc} 
\end{eqnarray*}
 for all $\tilde q \geq q$, 
with $ \tilde\beta = \beta - 3 
\left( \frac{1}{q} - \frac{1}{\tilde{q}} \right)$. 
Under this last condition, $q = \frac{12}{4\beta-1}$ 
is equivalent to $ \tilde q = \frac{12}{4\tilde\beta-1}$. 
It is then required that 
$ \frac{1}{\tilde{q}}+\frac{1}{q}=\frac{1}{2}$, 
which leads, because of the relations $ 
q = \frac{12}{4\beta-1}$, $ \tilde q = \frac{12}{4\tilde\beta-1}$, 
to the constraint $\tilde\beta=2-\beta$. Note that, when $\beta$ belongs 
to $(1,2)$, $\tilde\beta$ and $\beta-1$ belong to $(0,1)$; furthermore, 
in this case, we have $q < 4 < \tilde q$. Now, the remaining requirement from 
Lemma \ref{cruciallemma} is $\tilde\beta+(\beta-1)\geq1$, which is fulfilled 
(since $\tilde\beta+(\beta-1)=1$). Hence, 
we get that $ I^\eps_1  \rightarrow 0$ as $\eps \rightarrow 0 $.

Similarly,  one uses  Lemma \ref{cruciallemma} with the same functions  $\phi^1$, $\phi^2$, $\phi^3$, but this time with $i=0$ and $\partial_k \chi $ instead of $\chi$,
to get that $ I^\eps_2  \rightarrow 0$ as $\eps \rightarrow 0 $. Thus
\begin{equation}
\label{tri2}
\int_{(0,T) \times \R^3}  \chi  F^\eps_2 [m,H]  \, dx \, dt  \rightarrow 0,  
\text{ as } \eps \rightarrow 0 .
\end{equation}
\subparagraph{Third term.}
Now, 
\begin{equation*}
\int_{(0,T) \times \R^3}  \chi  F^\eps_3 [m,H]  \, dx \, dt = 
2   \sum_k  \int_{(0,T) \times \R^3}  \chi 
\mathcal{B}^{\eps} [m , \dt m_\eps -2 H_\eps] \cdot 
\partial_k   ( \partial_k m ) 
\, dx \, dt ,
\end{equation*}
so that we apply again Lemma \ref{cruciallemma} 
with $i=1$, $\phi^3 = \nabla m \in  \widetilde{L}^4 (0,T ; \dot{B}^{\beta -1}_{q,\infty} ( \R^3) )_{\rm loc}$, 
$\phi^2 = \dt m - 2 H \in L^2 ((0,T) \times  \R^3 )$ 
and $\phi^1 = m \in \widetilde{L}^4 (0,T ; \dot{B}^\beta_{q,\infty} ( \R^3) )_{\rm loc} \hookrightarrow 
\widetilde{L}^4 (0,T ; \dot{B}^{\tilde\beta}_{\tilde{q},\infty} ( \R^3) )_{\rm loc}$, exactly as for the 
second term $F^2_\eps  [m,H] $, to obtain 
\begin{equation}
\label{tri3}
\int_{(0,T) \times \R^3}  \chi  F^\eps_3 [m,H]  \, dx \, dt  \rightarrow 0,  
\text{ as } \eps \rightarrow 0 .
\end{equation}
\subparagraph{Fourth term.}
Finally we use again an integration by parts to get
\begin{eqnarray*}
\int_{(0,T) \times \R^3}  \chi  F^\eps_4 [m ]  \, dx \, dt &=& 
4 \sum_k   \int_{(0,T) \times \R^3}  \chi  \mathcal{B}^{\eps} [m , \partial_k  m]  
\cdot \Delta \partial_k m_\eps  \, dx \, dt
\\ &&+ 4 \sum_k   \int_{(0,T) \times \R^3}  (\partial_k \chi)  \mathcal{B}^{\eps} [m , \partial_k  m]  
\cdot \Delta  m_\eps  \, dx \, dt  .
\end{eqnarray*}
Let us start with the first term of the right hand side.
Here, we invoke Lemma \ref{cruciallemma} with $i=2$, so that the previous 
$\widetilde{L}^4 \dot{B}^\beta_{q,\infty}$-regularity for $m$ is useless. Instead, we take 
$\phi^1 = m \in \widetilde{L}^3 \dot{B}^\alpha_{p,c_0} \hookrightarrow 
\widetilde{L}^3 \dot{B}^{\tilde\alpha}_{\tilde p,c_0}$, 
$\phi^2 = \phi^3 = \nabla m \in \widetilde{L}^3 \dot{B}^{\alpha-1}_{p,c_0}$. 
We have the constraints 
$$
\tilde\alpha = \alpha - 3 \left( \frac{1}{p}-\frac{1}{\tilde p} \right) 
< \alpha, \quad \frac{1}{p}+\frac{2}{\tilde p} = 1 \quad \text{and} \quad 
\tilde\alpha + 2(\alpha-1) \geq 2. 
$$
Choosing $\tilde\alpha + 2(\alpha-1) = 2$ is equivalent to the relation 
$ p = \frac{9}{3\alpha-1}$. Furthermore, imposing 
$\tilde\alpha = 4 - 2\alpha \in (0,1)$ is equivalent to $\alpha\in(3/2,2)$. 
This is enough to ensure 
\begin{equation}
\label{tri4}
\int_{(0,T) \times \R^3}  \chi   F^\eps_4 [m ]  \, dx \, dt  \rightarrow 0,  
\text{ as } \eps \rightarrow 0 .
\end{equation}

Gathering
\eqref{tri1}-\eqref{tri2}-\eqref{tri3}-\eqref{tri4}
yields \eqref{tri}.

%%%%%%%%%%
%%%%%%%%%%
%%%%%%%%%%
%%%%%%%%%%
%%%%%%%%%%

\section{Vanishing of anomalous dissipations for the HMHD equations: 
Proof of Part ii) of Theorem \ref{Hallonsaghel} and Theorem~\ref{Hallonsagen}}
\label{Va2}
\subsection{No anomalous magneto-helicity dissipation: Proof of Part ii) 
of Theorem \ref{Hallonsaghel}}
In order to prove Part ii) of Theorem \ref{Hallonsaghel}, it is sufficient 
to  prove that for all $\chi \in C^\infty_{\rm c}((0,T)\times\R^3)$, 
\begin{equation}
\label{triBis}
\int_{(0,T) \times \R^3} \chi   d_m^{\mathfrak{a},\eps} \, dx \, dt  
\rightarrow 0 , \quad \text{as } \eps \rightarrow 0 .
\end{equation}
For the sequel, we fix such a function $\chi$. 
Let us  recall the definition of  $d_m^{\mathfrak{a},\eps}$ given in \eqref{eli}:
\begin{equation*}
d_m^{\mathfrak{a},\eps}  := 
2   A_{\eps} \cdot \curl    \mathcal{B}^{\eps} [u,B ]   
-  2   A_{\eps} \cdot  \curl     \Div   \mathcal{C}^{\eps} [B,B] =: 2 T_1^{\eps } [u,A] -2 T_2^{\eps } [A] .  
\end{equation*}

We use 
\eqref{formu} and an integration by parts to obtain:
\begin{eqnarray}
\label{triter}
\int_{(0,T) \times \R^3} \chi  T_1^{\eps } [u,A]   \, dx \, dt  
&=& \int_{(0,T) \times \R^3} \chi  B_{\eps} \cdot   \mathcal{B}^{\eps} [u,B ]  \, dx \, dt  
\\ \nonumber&& - \int_{(0,T) \times \R^3} \big(  \mathcal{B}^{\eps} [u,B ] \wedge A_{\eps} ) \cdot \nabla \chi   \, dx \, dt  .
\end{eqnarray}

On the other hand, we apply \eqref{formu} to get:
\begin{eqnarray}
\nonumber
\int_{(0,T) \times \R^3} \chi  T_2^{\eps } [A]   \, dx \, dt  
&=&  
\int_{(0,T) \times \R^3} \chi  \curl   A_{\eps} \cdot    \Div   \mathcal{C}^{\eps} [B,B]  \, dx \, dt
\\ \nonumber&& + \int_{(0,T) \times \R^3} \chi    \Div \Big(  ( \Div   \mathcal{C}^{\eps} [B,B] ) \wedge   A_{\eps}   \Big) \, dx \, dt .
\end{eqnarray}
An integration by parts yields 
\begin{eqnarray}
 \nonumber 
 \int_{(0,T) \times \R^3} \chi  T_2^{\eps } [A]   \, dx \, dt  
 &=& -  \int_{(0,T) \times \R^3}  \chi   ( \nabla  \curl   A_{\eps} ) :  \mathcal{C}^{\eps} [B,B]  \, dx \, dt 
\\ \nonumber&&- \int_{(0,T) \times \R^3} (  \mathcal{C}^{\eps} [B,B]    \curl   A_{\eps} )  \cdot \nabla   \chi   \, dx \, dt
\\ \nonumber && - \int_{(0,T) \times \R^3} \big( ( \Div \mathcal{C}^{\eps} [B,B]  )   \wedge   A_{\eps} \big)  \cdot \nabla   \chi   \, dx \, dt ,
\end{eqnarray}
so that we  obtain finally:
\begin{eqnarray}
 \nonumber
 \int_{(0,T) \times \R^3} \chi  T_2^{\eps } [A]   \, dx \, dt  
&=& -  \int_{(0,T) \times \R^3}  \chi   
( \nabla  B_{\eps} ) :  \mathcal{C}^{\eps} [B,B]  \, dx \, dt 
\\ \nonumber&&- \int_{(0,T) \times \R^3} 
(  \mathcal{C}^{\eps} [B,B] B_{\eps} )  \cdot \nabla   \chi   \, dx \, dt
\\ \nonumber && + \int_{(0,T) \times \R^3}  
\Big( ( \mathcal{C}^{\eps} [B,B] \nabla  ) \wedge   
A_{\eps} \Big)  \cdot \nabla   \chi   \, dx \, dt 
\\ \label{posspo}
&& + \sum_{j} \int_{(0,T) \times \R^3}   
\Big(  ((\mathcal{C}^{\eps} [B,B  ])_{ij})_{i} \wedge 
A_{\eps} \Big) \cdot \partial_{j} \nabla \chi  \, dx \, dt .
\end{eqnarray}
To get the vanishing of the $T_1^\eps$ term given by \eqref{triter} 
as $\eps$ goes to zero, observe that the regularity 
given by Theorem \ref{ADFL} and interpolation theory 
suffices, applying Lemma \ref{cruciallemma} with $i=0$ 
and the $\phi^j$'s equal to $u$, $B$ or $A$, all belonging 
to $L^\infty(0,T;L^2(\R^3)) \cap L^2(0,T;H^1(\R^3)) 
\hookrightarrow L^3(0,T;L^3(\R^3))$. 

Concerning the last three terms produced by $T^\eps_2$ in \eqref{posspo}, 
use again Lemma \ref{cruciallemma} with $i=0$, $\phi^1=\phi^2=B$, 
and $\phi^3$ being $B$, $\nabla A$ or $A$, which all belong to 
$L^3(0,T;L^3(\R^3))$. The first term in \eqref{posspo} is the one 
for which the most regularity is needed: in Lemma \ref{cruciallemma}, 
take $i=1$ and $\phi^1=\phi^2=\phi^3=B \in 
\widetilde{L}^3 (0,T;\dot{B}^{1/3}_{3,c_0}(\R^3))_{\rm loc}$.

%%%%%%%%%%
%%%%%%%%%%
%%%%%%%%%%
%%%%%%%%%%
%%%%%%%%%%

\subsection{No anomalous energy dissipation: Proof of Part ii) 
of Theorem \ref{Hallonsagen}}

In order to prove Part ii)  of Theorem \ref{Hallonsagen}, 
we consider $\chi \in C^\infty_{\rm c}((0,T)\times\R^3)$, 
and we prove that 
\begin{equation}
\label{triFin}
\int_{(0,T) \times \R^3} \chi  d_{\rm HMHD}^{\mathfrak{a},\eps}  \, dx \, dt  
\rightarrow 0 , \quad \text{as } \eps \rightarrow 0 .
\end{equation}
We recall that 
\begin{gather*}
d_{\rm HMHD}^{\mathfrak{a},\eps}   :=
- u_\eps \cdot   \Div  \Big( \mathcal{C}^{\eps} [u,u ] 
- \mathcal{C}^{\eps} [B,B  ]     \Big) 
-   \frac{1}{2} u_\eps \cdot  \nabla {\mathcal A}^\eps [B, B ]
\\  - B_{\eps} \cdot \Big(  \curl   \mathcal{B}^{\eps} [u,B ] \Big)
+ B_{\eps} \cdot  \Big(  \curl   \Div \mathcal{C}^{\eps} [B,B  ]   \Big)  
\\  =:  J^{\eps}_1 [u,B] + J^{\eps}_2 [u,B] + J^{\eps}_3 [u,B] + J^{\eps}_4 [B] .
\end{gather*}
Using \eqref{formu} and that  $u_{\eps}$ is divergence free, we have 
\begin{eqnarray}
\nonumber
\int_{(0,T) \times \R^3} \chi  J^{\eps}_1 [u,B]   \, dx \, dt   
&=& \int_{(0,T) \times \R^3} \chi   \nabla   u_{\eps} \cdot   
\Big( \mathcal{C}^{\eps} [u,u ] - \mathcal{C}^{\eps} [B,B  ] \Big) \, dx \, dt   
\\ \nonumber && + \int_{(0,T) \times \R^3}
\Big( \mathcal{C}^{\eps} [u,u ] - \mathcal{C}^{\eps} [B,B] \Big)  
(   u_{\eps} \cdot     \nabla \chi )  \, dx \, dt    ,
\end{eqnarray}
\begin{eqnarray}
  \nonumber   \int_{(0,T) \times \R^3} \chi  J^{\eps}_2 [u,B] \, dx \, dt   
&=&   \frac{1}{2}  \int_{(0,T) \times \R^3}  (u_\eps \cdot \nabla \chi )  
{\mathcal A}^\eps [B, B ]  \, dx \, dt    ,
\end{eqnarray}
\begin{eqnarray}
 \nonumber   \int_{(0,T) \times \R^3} \chi  J^{\eps}_3 [u,B]  \, dx \, dt    
&=& - \int_{(0,T) \times \R^3} \  
\chi ( \curl  B_{\eps}) \cdot \mathcal{B}^{\eps} [u,B ]  \, dx \, dt 
\\ \nonumber && - \int_{(0,T) \times \R^3} \  
\Big( B_{\eps} \wedge  \mathcal{B}^{\eps} [u,B ] \Big) \cdot 
\nabla \chi   \, dx \, dt ,
\end{eqnarray}
and
\begin{eqnarray}
  \nonumber  \int_{(0,T) \times \R^3} \chi  J^{\eps}_4 [B] \, dx \, dt   
 &=& \int_{(0,T) \times \R^3} \chi \Div \Big(  (\Div \mathcal{C}^{\eps} [B,B] ) 
\wedge B_{\eps} \Big)  \, dx \, dt
 \\  \nonumber &&+ \int_{(0,T) \times \R^3} \chi  \Big(  (\Div \mathcal{C}^{\eps} [B,B  ]  ) 
 \cdot   \curl  B_{\eps} \Big)  \, dx \, dt ,
 \end{eqnarray}
by using again  \eqref{formu}.

Then, integrating by parts:
\begin{eqnarray}
  \nonumber 
  \int_{(0,T) \times \R^3} \chi  J^{\eps}_4 [B] \, dx \, dt   
  &=& - \int_{(0,T) \times \R^3}   
\Big( (\Div \mathcal{C}^{\eps} [B,B  ]  ) \wedge B_{\eps} \Big) \cdot  
\nabla \chi  \, dx \, dt
% \\ &&  
\\  \nonumber && - \int_{(0,T) \times \R^3} \chi  
 \nabla  ( \curl  B_{\eps})   : \mathcal{C}^{\eps} [B,B  ]   \, dx \, dt
\\  \nonumber && - \int_{(0,T) \times \R^3} \   \Big( \mathcal{C}^{\eps} [B,B  ]  \curl  B_{\eps} \Big)   \cdot   \nabla  \chi ,
\end{eqnarray}
so that
\begin{eqnarray}
 \nonumber &=& \int_{(0,T) \times \R^3}   
 \Big(  (\mathcal{C}^{\eps} [B,B] \nabla) \wedge B_{\eps} \Big) \cdot  
 \nabla \chi  \, dx \, dt
 %\\ && 
\\  \nonumber && + \sum_{j} \int_{(0,T) \times \R^3}   \Big(  ((\mathcal{C}^{\eps} [B,B  ])_{ij})_{i} \wedge B_{\eps} \Big) \cdot \partial_{j} \nabla \chi  \, dx \, dt
\\ &&  \label{formub}
 - \int_{(0,T) \times \R^3} 
\chi  \nabla  ( \curl  B_{\eps})   : \mathcal{C}^{\eps} [B,B  ]   \, dx \, dt
%\\ && 
\\  \nonumber && - \int_{(0,T) \times \R^3} \   \Big( \mathcal{C}^{\eps} [B,B  ]  
\curl  B_{\eps} \Big)   \cdot   \nabla  \chi   \, dx \, dt .
\end{eqnarray}

Finally, we use repetitively Lemma \ref{cruciallemma} 
with the $\phi^j$'s equal to $u$ or $B$, 
and $i=0$ ($u,B \in L^3(0,T;L^3(\R^3))$), 
$i=1$ ($u,B \in \widetilde{L}^3 (0,T;\dot{B}^{1/3}_{3,c_0}(\R^3))_{\rm loc}$), 
or $i=2$ 
($u,B \in \widetilde{L}^3 (0,T;\dot{B}^{2/3}_{3,c_0}(\R^3))_{\rm loc}$): 
the worse term to cope with, in term of needed regularity, 
is $\chi  \nabla  ( \curl  B_{\eps}) : \mathcal{C}^{\eps} [B,B]$, 
from \eqref{formub}. 
It then follows that, as $\eps \rightarrow 0$, 
\begin{equation*}
\text{for } i=1,2,3, \quad 
\int_{(0,T) \times \R^{3}} \chi  J^{\eps}_i [u,B]  \, dx \, dt \rightarrow 0,  
\quad \text{and} \quad 
\int_{(0,T) \times \R^{3}} \chi  J^{\eps}_4 [B]    \, dx \, dt  \rightarrow 0,  \end{equation*}
which provides \eqref{triFin}.

%%%%%%%%%%
%%%%%%%%%%
%%%%%%%%%%
%%%%%%%%%%
%%%%%%%%%%

\section{Vanishing of anomalous crossed fluid-magneto-helicity dissipation for the MHD equations:  Proof of Part iii) of Theorem \ref{MHDonsag}}
\label{pom}

We consider the inner product of equation \eqref{HMHD1Qsplit} with $B_{\eps}$ and the  inner product of equation \eqref{HMHD3Qsplit} with $u_{\eps}$ and we sum the resulting identities, taking into account that
\begin{gather*}
 B_{\eps} \cdot \Div (u_{\eps} \otimes u_{\eps} - B_{\eps} \otimes     B_{\eps}) 
 - u_{\eps} \cdot \curl  (u_{\eps} \wedge  B_{\eps}) 
 = \Div \Big( (u_{\eps} \cdot B_{\eps} ) u_{\eps} - \frac12 (| u_{\eps} |^2 + | B_{\eps} |^2 )  B_{\eps} \Big) ,
 \\  B_{\eps} \cdot \nabla (p_{\eps} + \frac12  | B_{\eps} |^2 )  = \Div \big( (p_{\eps} + \frac12  | B_{\eps} |^2 ) B_{\eps}  \big) ,
 \\  - B_{\eps} \cdot  \Delta u_{\eps}  -  u_{\eps} \cdot   \Delta B_{\eps} = 2 (\curl  u_{\eps} ) \cdot (\curl  B_{\eps} )  + \Div \big(  (\curl  u_{\eps} ) \wedge B_{\eps} \big) + \Div \big(  (\curl  B_{\eps} ) \wedge u_{\eps} \big) .
\end{gather*}
We obtain that 
\begin{gather*}
\dt h^{\eps}_{fm} 
+ d^{\eps}_{fm}  
 + \Div   f^{\eps}_{fm} 
  =  d^{\mathfrak{a},\eps}_{fm} , 
\end{gather*}
where
\begin{gather*}
h^{\eps}_{fm} := u_{\eps} \cdot B_{\eps} ,
\quad d^{\eps}_{fm} :=     2 \omega_{\eps} \cdot \curl B_{\eps}                        ,
\\ f^{\eps}_{fm} :=  (u_{\eps}\cdot B_{\eps})u_{\eps} + (p_{\eps} -\frac12 |u_{\eps} |^{2} )B_{\eps} + (\curl u_{\eps} ) \wedge B_{\eps} +  (\curl B_{\eps} ) \wedge u_{\eps}  ,
\\ d^{\mathfrak{a},\eps}_{fm}  := 
- B_{\eps} \cdot \Div \Big(  \mathcal{C}^{\eps} [u,u  ] -  \mathcal{C}^{\eps} [B,B  ]\Big) 
- \frac12 B_{\eps} \cdot  \nabla  \mathcal{A}^{\eps} [B,B ]    
+ u^{\eps } \cdot \curl   \mathcal{B}^{\eps} [u,B ]  .  
\end{gather*}
One easily sees that 
$h^{\eps}_{fm} $, 
$d^{\eps}_{fm}$, 
and $f^{\eps}_{fm}$ converge respectively in the sense of distributions to 
$h_{fm}$, 
$d_{fm}  $ and 
 $ f_{fm}$.

Finally, we use repetitively Lemma \ref{cruciallemma} to conclude that the anomalous dissipation $d^{\mathfrak{a},\eps}_{fm}$ vanishes when $\eps \rightarrow 0$.

%%%%%%%%%%%%%%%%%%%%
%%%%%%%%%%%%%%%%%%%%SuitableMLL
%%%%%%%%%%%%%%%%%%%%

\section{Suitable solutions: Proof of Theorem  \ref{THsuitableMLL} 
and of Theorem \ref{THsuitableHMHD}}
\label{SectionSuitable}

\subsection{Proof of Theorem  \ref{THsuitableMLL}}

We recall from \cite{AS,CF98} that the weak maximum principle yields that for any $\eps \in (0,1)$,  $ m^{\eps}$ is bounded by $1$ almost everywhere in space and time. 
Then, using equation \eqref{LLM-suit1}  and the estimate \eqref{MLLenergyGL} we deduce that $m^{\eps}$ belongs to
the space $L^{2} ((0,T) ; H^{2} (\R^3 ; \R^3))   $ (observe however that this provides an estimate of the norm of $m^{\eps}$  in this space which is not uniform in $\eps$).  

Now, observe that, formally,  multiplying  \eqref{LLM-suit1} by $\dt m^{\eps}$, \eqref{LLM-suit2} by $H^{\eps}$, \eqref{LLM-suit3} by $E^{\eps}$, 
and summing the resulting identities lead to 
\begin{equation}
\label{LLlocalEpsG}
\dt \big( e^{\eps}_{\rm MLL} +   e^{\eps}_{\rm GL}       \big) + 
d^{\eps}_{\rm MLL}  + 
\Div f^{\eps}_{\rm MLL}   = - 2 (m^{\eps} \cdot H^{\eps} ) (m^{\eps} \cdot \partial_{t} m^{\eps} ) ,
\end{equation}
where $(e^{\eps}_{\rm MLL} , d^{\eps}_{\rm MLL}  ,  f^{\eps}_{\rm MLL}  )$ is given by \eqref{edfMLLeps} and  $e^{\eps}_{\rm GL}$ by  \eqref{88}.

In fact,  Identity  \eqref{LLlocalEpsG} can be easily justified. On one hand  the smoothness of $(m^{\eps} , H^{\eps} , E^{\eps})$ is sufficient to manipulate all the terms coming from the multiplication of  \eqref{LLM-suit1} by $\dt m^{\eps}$. On the other hand, one can mollify the linear equations \eqref{LLM-suit2}  and \eqref{LLM-suit3}  and then multiply them respectively by 
$H^{\eps}$ and $E^{\eps}$. It then remains to sum the resulting identities and to pass to the limit with respect to the regularization parameter. 
Thus, the local energy identity \eqref{LLlocalEpsG} holds true without any anomalous dissipation.

Now, using the uniform bounds of  $(m^{\eps} , H^{\eps} , E^{\eps})$  provided by \eqref{MLLenergyGL}, we have, 
up to a subsequence: 
$ f^{\eps}_{\rm MLL}  $  converges in the sense of distributions to 
 $f_{\rm MLL}  $;
$d^{\eps}_{\rm MLL}  $ and $e^{\eps}_{\rm MLL} +  e^{\eps}_{\rm GL}$ 
 converge in the sense of distributions respectively to 
 some distributions  $\tilde{d}_{\rm MLL}^{\mathfrak{a},1}$ and $\tilde{e}_{\rm MLL}^{\mathfrak{a}}$ 
 such that, by lower weak semicontinuity,
  $ \tilde{d}_{\rm MLL}^{\mathfrak{a},1} - |\dt m |^{2} $ and
  $\tilde{e}_{\rm MLL}^{\mathfrak{a}} - e_{\rm MLL}$ 
  are non negative. 
Finally the term $(m^{\eps} \cdot H^{\eps} ) (m^{\eps} \cdot \partial_{t} m^{\eps} )$ converges in the sense of distributions to $0$.

Therefore it follows from \eqref{LLlocal} that 
the  anomalous energy dissipation  $d^{\mathfrak{a}}_{\rm MLL}$ is given by 
$d^\mathfrak{a}_{\rm MLL} = - d_{\rm MLL}^{\mathfrak{a},1}  - \partial_t e_{\rm MLL}^{\mathfrak{a}}$, 
where $d_{\rm MLL}^{\mathfrak{a},1}  := \tilde{d}_{\rm MLL}^{\mathfrak{a},1}  - d_{\rm MLL} $
and
$e_{\rm MLL}^{\mathfrak{a}} := \tilde{e}_{\rm MLL}^{\mathfrak{a}} - e_{\rm MLL}$
 are non negative.
%

%%%%%%%%%%%%%%%%%%%%
%%%%%%%%%%%%%%%%%%%%
%%%%%%%%%%%%%%%%%%%%

\subsection{Proof of Theorem \ref{THsuitableHMHD}}

Multiplying \eqref{HMHD-suit1} by $u^{\eps}$, \eqref{HMHD-suit3} by $B^{\eps}$, and summing the resulting identities lead to 
\begin{equation}
\dt e^{\eps}_{\rm HMHD} + 
d^{\eps}_{\rm HMHD}  + 
\Div \tilde{f}^{\eps}_{\rm HMHD}  = 0  ,
\end{equation}
where $e^{\eps}_{\rm HMHD}$ and $ 
d^{\eps}_{\rm HMHD}$ are given by the formula in \eqref{defcucu}
and
\begin{gather*}
\tilde{f}^{\eps}_{\rm HMHD}  := 
  \frac12 |  u^{\eps} |^2 ( u^{\eps})_{\eps} + p^{\eps} u^{\eps}
 + B^{\eps} \wedge  ( u^{\eps} \wedge (B^{\eps})_{\eps} )
 + ( \curl u^{\eps} )  \wedge (u^{\eps})_{\eps}
 \\ +  ( \curl B^{\eps} )  \wedge (B^{\eps})_{\eps}
+  \big(  ( \curl B^{\eps} )  \wedge (B^{\eps})_{\eps} \big) \wedge B^{\eps} .
\end{gather*}
We do not detail the computations here since it is sufficient to adapt what we have already done in 
Section \ref{LoEnHMHD} in order to obtain \eqref{cucu} with here an extra bookkeeping of the   mollifications.

Now, using that, up to a subsequence,    $u^{\eps}$ and $ B^{\eps}$ converges in $L^{3} ((0,T) \times \R^{3} )_{\rm loc}$ respectively to $u$ and $B$, and therefore $(u^{\eps})_{\eps}$ and $ (B^{\eps})_{\eps}$ converges as well in $L^{3} ((0,T) \times \R^{3} )_{\rm loc}$ respectively to $u$ and $B$, and that   $ \curl u^{\eps}$ and $ \curl B^{\eps}$ weakly converge to $ \curl u$ and $ \curl B$ in $L^{2} ((0,T) \times \R^{3} )$, 
we get that 
\begin{gather*}
\frac{1}{2} \dt \Big( |u^{\eps} |^{2} +   |B^{\eps} |^{2}   \Big) +
 \Div  \Big(
 ( \frac12 |  u^{\eps} |^2 ( u^{\eps})_{\eps} + B^{\eps} \wedge  ( u^{\eps} \wedge (B^{\eps})_{\eps} )
 + ( \curl u^{\eps} )  \wedge (u^{\eps})_{\eps}
 +  ( \curl B^{\eps} )  \wedge (B^{\eps})_{\eps}  \Big)
\end{gather*}
converges in the sense of distributions to
\begin{gather*}
\frac{1}{2} \dt \Big( |u |^{2} +   |B |^{2}   \Big) +
   \Div  \Big(
  \frac12 |  u |^2  u
 + B \wedge  ( u \wedge B )
 + ( \curl u )  \wedge u
 +  ( \curl B )  \wedge B
\Big) .
\end{gather*}

Moreover, using once again \eqref{ellpm}, we get that $ p_{\eps} $ converges in  $L^\frac32 ((0,T) \times \R^3 )_{\rm loc}$
to $p$. Thus 
$ \Div  ( p^{\eps} u^{\eps} )$ converges in the sense of distributions to 
 $ \Div ( p u)$.

Finally, using  that, up to a subsequence,   $ B^{\eps}  $ converges to $B$ in $L^{4} ((0,T) \times \R^{3} )_{\rm loc}$ we obtain that 
\begin{equation}
\label{hallli}
\Div  \big(  ( \curl B^{\eps} )  \wedge (B^{\eps})_{\eps} \big) \wedge B^{\eps}
\big)
\end{equation}
converges in the sense of distributions to
$\Div  \big(  ( \curl B )  \wedge B \big) \wedge B\big)$.

We therefore get that 
 the anomalous energy dissipation $d_{\rm HMHD}^\mathfrak{a} $ is given by the discrepancy between the limit of $   |\curl u^{\eps}  |^2 + |\curl B^{\eps}  |^2$ and
    $|\curl u  |^2 + |\curl B  |^2$ which is 
 non positive by lower weak semicontinuity.

Since the term 
\eqref{hallli}
vanishes when the Hall effect is omitted one sees that the assumption  that, up to a subsequence,   $ B^{\eps}  $ converges to $B$ in $L^{4} ((0,T) \times \R^{3} )_{\rm loc}$, is not needed in order to prove the result for the MHD equations.

%%%%%%%%%%
%%%%%%%%%%
%%%%%%%%%%
%%%%%%%%%%
%%%%%%%%%%

\appendix

\bigskip
\noindent
{\bf \Large Appendix: proof of Lemma \ref{bern}}

\medskip
\noindent

Here, we prove that for all $p,r\in[1,\infty]$, $\alpha\in(0,1)$ 
and $\tilde\alpha\in(0,\alpha)$, with $\tilde{p}$ defined by 
$\tilde{\alpha} - 3/\tilde{p} = \alpha -3/p$, the space 
$\widetilde{L}^r (0,T; \dot{B}^\alpha_{p,\infty} (\R^3))$ 
is continuously embedded in 
$\widetilde{L}^r (0,T; \dot{B}^{\tilde{\alpha}}_{\tilde{p},\infty} (\R^3))$. 

The proof of the other case, when $\alpha\in(1,2)$ 
and $\tilde\alpha\in(0,1)\cup (1,\alpha)$, follows, since, when $\alpha\in(1,2)$,
$$\| u  \|_{\widetilde{L}^r (0,T; \dot{B}^\alpha_{p,\infty} (\R^3))} = 
\| u  \|_{\widetilde{L}^r (0,T; \dot{B}^{\alpha -1}_{p,\infty} (\R^3))}
+ \sum_{i} \| \partial_{i} u  \|_{\widetilde{L}^r (0,T; \dot{B}^{\alpha -1}_{p,\infty} (\R^3))} 
.$$

As usual, $A \lesssim B$ denotes the inequality 
$A \leq CB$ for some universal constant $C$. 
We recall the existence of a smooth dyadic partition of unity:  
there exists a smooth radial function $\varphi$, 
supported in the annulus 
$ C(3/4,8/3) := \{ 3/4 < | \xi | < 8/3 \}$, 
with values in the interval $\lbrack 0,1 \rbrack$,
such that 
$$
\forall \xi \in \R^3 \setminus \{0\} ,\quad   
\sum_{ j \in \Z} \varphi  (2^{-j} \xi ) = 1 ; \qquad 
| j-j' |  \geqslant 2 \Rightarrow 
\text{ supp }  \varphi  (2^{-j}  \cdot  ) \cap   
\text{ supp } \varphi  (2^{-j'} \cdot) =  \emptyset.
$$
The so-called dyadic blocks $\dot\Delta_j $ correspond to 
the Fourier multipliers 
$ \dot\Delta_j :=  \varphi (2^{-j} D) $, that is
\begin{eqnarray*}
\dot\Delta_j u (x):=  
2^{3j} \int_{\R^3}  h (2^{j}y) u(x-y) dy  
\quad \text{ for }  j\in\Z , \quad \text{ where }   
h :=  \mathcal{F}^{-1}  \varphi .
\end{eqnarray*}
Then, for all $u\in\mathcal{S}'_h$ (from Definition \ref{oula}), 
homogeneous  Littlewood-Paley decomposition holds: 
$u = \sum_{j\in\Z} \dot\Delta_j u$. 
    
Let $u \in \widetilde{L}^r (0,T; \dot{B}^\alpha_{p,\infty} (\R^3))$ 
and $y\in\R^3\setminus\{0\}$.  
Denoting $\delta_y u(t,x) = u(t,x-y) - u(t,x)$, 
we write 
$$
\delta_y u = \sum_{j\in\Z} \dot\Delta_j \delta_y u . 
$$
Thus, we estimate 
$$
\| \| \delta_y u \|_{L^{\tilde{p}}(\R^3)} \|_{L^r(0,T)} 
\leq \sum_{j\in\Z} 
\| \| \delta_y \dot\Delta_j u \|_{L^{\tilde{p}}(\R^3)} \|_{L^r(0,T)} , 
$$
which splits up, for any $j_y\in\Z$, as the sum $I+II$, 
where 
$$
I = \sum_{j\leq j_y} 
\| \| \delta_y \dot\Delta_j u \|_{L^{\tilde{p}}(\R^3)} \|_{L^r(0,T)} 
\quad \text{and} \quad 
II = \sum_{j>j_y} 
\| \| \delta_y \dot\Delta_j u \|_{L^{\tilde{p}}(\R^3)} \|_{L^r(0,T)}. 
$$

We recall that Bernstein's lemma implies
\begin{equation} \label{bernineq}
\| \dot\Delta_j u \|_{L^{\tilde{p}}(\R^3)} 
\lesssim 2^{3j\left(\frac{1}{p}-\frac{1}{\tilde{p}}\right)} 
\| \dot\Delta_j u \|_{L^p(\R^3)} 
= 2^{3j(\alpha-\tilde\alpha)} \| \dot\Delta_j u \|_{L^p(\R^3)} .
\end{equation}
We also notice that 
\begin{equation} \label{triangineq}
\| \delta_y \dot\Delta_j u \|_{L^{\tilde{p}}(\R^3)} 
\leq 2 \| \Delta_j u \|_{L^{\tilde{p}}(\R^3)},
\end{equation}
as well as 
\begin{equation} \label{taylorineq}
\| \delta_y \dot\Delta_j u \|_{L^{\tilde{p}}(\R^3)} 
\lesssim 2^j |y| \sum_{|j-j'|\leq1} 
\| \Delta_{j'} u \|_{L^{\tilde{p}}(\R^3)}
\end{equation}
(see \cite{bcd}, page 75), and 
\begin{equation} \label{freqineq}
2^{j\alpha} 
\| \| \delta_y \dot\Delta_j u \|_{L^{\tilde{p}}(\R^3)} \|_{L^r(0,T)} 
\lesssim \| u \|_{\widetilde{L}^r (0,T; \dot{B}^\alpha_{p,\infty} (\R^3))} 
\end{equation}
(from \cite{bcd}, page 76).

Then, for any term in the sum $I$, we have 
\begin{eqnarray}
\notag
\| \| \delta_y \dot\Delta_j u \|_{L^{\tilde{p}}(\R^3)} \|_{L^r(0,T)}
& \lesssim & 
2^j |y| \sum_{|j-j'|\leq1} 
\| \| \Delta_{j'} u \|_{L^{\tilde{p}}(\R^3)} \|_{L^r(0,T)}
\quad \text{by \eqref{taylorineq}}, \\
\notag
& \lesssim & 
2^j |y| 2^{3j(\alpha-\tilde\alpha)} 
\sum_{|j-j'|\leq1} \| \| \Delta_{j'} u \|_{L^p(\R^3)} \|_{L^r(0,T)}
\quad \text{by \eqref{bernineq}}, \\ 
& \lesssim & |y| 2^{j(1-\tilde\alpha)} 
\| u \|_{\widetilde{L}^r (0,T; \dot{B}^\alpha_{p,\infty}(\R^3))} 
\quad \text{by \eqref{freqineq}}.
\label{ineqI}
\end{eqnarray}

Now, for the terms in the sum $II$, inequalities 
\eqref{triangineq}, \eqref{bernineq} and \eqref{freqineq} 
lead to 
\begin{equation} \label{ineqII}
\| \| \delta_y \dot\Delta_j u \|_{L^{\tilde{p}}(\R^3)} \|_{L^r(0,T)} 
\lesssim 2^{-j\tilde\alpha} 
\| u \|_{\widetilde{L}^r (0,T; \dot{B}^\alpha_{p,\infty}(\R^3))} . 
\end{equation}

Finally, choosing $j_y\in\Z$ such that 
$\displaystyle \frac{1}{|y|} \leq 2^{j_y} \frac{2}{|y|}$, 
we get 
$\displaystyle 
|y| \sum_{j\leq j_y} 2^{j(1-\tilde\alpha)} 
+ \sum_{j> j_y} 2^{-j\tilde\alpha} \lesssim |y|^{\tilde\alpha}$, 
and hence the desired inequality, 
$$
\| u \|_{\widetilde{L}^r (0,T; 
\dot{B}^{\tilde\alpha}_{\tilde{p},\infty}(\R^3))}
\lesssim
\| u \|_{\widetilde{L}^r (0,T; \dot{B}^\alpha_{p,\infty}(\R^3))} .
$$

\end{document}